\documentclass[11pt]{article}

\usepackage{amsmath, amssymb}
\usepackage{geometry}
\usepackage{microtype}
\usepackage{amsthm}
\usepackage{mathtools}
\usepackage[most]{tcolorbox}
\usepackage{latexsym}
\usepackage{graphicx}
\usepackage{setspace}
\usepackage{authblk}
\usepackage{pbox}
\usepackage{tikz}
\usepackage{faktor}
\usetikzlibrary{calc}
\usepackage{centernot}
\usepackage[utf8]{inputenc}
\usepackage[english]{babel}
\usepackage{lineno}
\usepackage{newtx}
\usepackage{pdfpages}
\usepackage[hidelinks]{hyperref}
\usepackage{xurl}
\usepackage{listings}
\usepackage{xcolor}
\usepackage{csquotes}
\usepackage{url}
\usepackage{caption}
\usepackage{subcaption}
\usepackage{gensymb}
\usepackage{array}
\usepackage{fix-cm}

\hypersetup{
    colorlinks=true,
    allcolors=blue,
    pdftitle={Overleaf Example},
    pdfpagemode=FullScreen,
    }

\usepackage[backend=biber, style=authoryear, sorting=nyt, giveninits=true, maxbibnames=6, minbibnames=1, maxcitenames=2, mincitenames=1, sortcites=true,uniquename=init]{biblatex}
\DeclareFieldFormat[article]{title}{#1}
\DeclareFieldFormat[book]{title}{#1}
\DeclareFieldFormat[article]{pages}{#1}
\DeclareFieldFormat{doi}{%
  \ifhyperref
    {\href{#1}{\nolinkurl{#1}}}
    {\nolinkurl{#1}}%
}

\DeclareBibliographyDriver{article}{%
  \printnames{author}%
  \space
  \bibopenparen\printfield{year}\bibcloseparen
  \space
  \printfield[title]{title}%
  \adddot\space
  \printfield{journaltitle}%
  \adddot\space
  \printfield{volume}%
  \addcolon
  \printfield[pages]{pages}%
  \newunit\newblock
  \iffieldundef{doi}
    {}
    {\space \printfield[doi]{doi}}%
  \adddot
}

\DeclareBibliographyDriver{book}{%
  \printnames{author}%
  \space%
  \bibopenparen\printfield{year}\bibcloseparen%
  \space%
  \printfield[title]{title}%
  \addperiod\space%
  \iffieldundef{edition}{}{\printfield{edition}\space}%
  \printlist{publisher}%
  \adddot\space%
  \printlist{location}%
  \space%
  \space\printfield{pages}%
  \addperiod%
}

\DeclareBibliographyDriver{incollection}{%
  \printnames{author}%
  \addperiod\space%
  \bibopenparen\printfield{year}\bibcloseparen%
  \addperiod\space%
  \printfield{title}%
  \addperiod\space%
  In:\space%
  \printfield{booktitle}%
  \addperiod\space%
  \iffieldundef{edition}{}{\printfield{edition}\space}%
  \printnames{editor}%
  \space(ed[s].)%
  \addperiod\space%
  \printlist{publisher}%
  \addcomma\space%
  \printlist{location}%
  \addcomma\space%
  \space\printfield{pages}%
  \addperiod%
}

\DeclareBibliographyDriver{online}{%
  \printnames{author}%
  \addperiod\space%
  \bibopenparen\printfield{year}\bibcloseparen%
  \space%
  \printfield{title}%
  \addperiod\space%
  Available at:\space\printfield{url}%
  \addperiod\space%
  Last accessed \printfield{urldate}%
  \addperiod%
}

\renewbibmacro*{journal}{%
  \printfield{shortjournal}}
\renewbibmacro*{date}{%
  \bibopenparen\printfield{year}\bibcloseparen%
}

\AtEveryBibitem{%
  \clearfield{issn}%
  \clearfield{isbn}%
  \clearfield{url}%
  \clearfield{number}%
  \clearfield{month}%
}

\DeclareNameAlias{author}{family-given}

\usepackage{lipsum}

\lstdefinestyle{mystyle}{
    language=Python,
    basicstyle=\ttfamily,
    numbers=left,
    numberstyle=\small,
    stepnumber=1,
    numbersep=5pt,
    backgroundcolor=\color{white},
    showspaces=false,
    showstringspaces=false,
    showtabs=false,
    tabsize=4,
    captionpos=b,
    breaklines=true,
    breakatwhitespace=true,
    title=\lstname,
    escapeinside={\%*}{*)},
    morekeywords={*,...}
}
\usepackage{booktabs}
\newcolumntype{C}[1]{>{\centering\arraybackslash}p{#1}}

\renewcommand*{\geq}{\geqslant}
\renewcommand*{\leq}{\leqslant}

\newtheorem{theorem}{Theorem}[section]

\newtheorem{lemma}{Lemma}[section]
\newtheorem{remark}{Remark}[section]

\lstdefinestyle{mystyle}{
    breakatwhitespace=false,         
    breaklines=true,                 
    captionpos=b,                    
    keepspaces=true,                 
    numbers=left,                    
    numbersep=5pt,                  
    showspaces=false,                
    showstringspaces=false,
    showtabs=false,                  
    tabsize=2}

\begin{document}


\begin{titlepage}
\begin{flushleft}

\textbf{\LARGE{Decoding the spatial spread of cyanobacterial blooms in an epilimnion}}

\vspace{0.5cm}



\textbf{Authors:}

Jacob Serpico$^{1}$, Kyunghan Choi$^{1}$, B. A. Zambrano-Luna$^{1}$, Tianxu Wang$^{1}$, Hao Wang$^{1,*}$

\vspace{0.5cm}

\textbf{Affiliations:}

$^1$Interdisciplinary Lab for Mathematical Ecology and Epidemiology \& Department of Mathematical and Statistical Sciences, University of Alberta, Edmonton, Alberta T6G 2R3, Canada

$^*$Corresponding author. Email: \href{mailto:hao8@ualberta.ca}{hao8@ualberta.ca}.

\vspace{0.5cm}

\textbf{Author Contributions:} \newline
\textbf{Jacob Serpico:} Conceptualization, methodology, software, validation, formal analysis, data curation, writing, visualization, project administration. 
\textbf{Kyunghan Choi:} Methodology, software, validation, formal analysis, writing.
\textbf{B.A. Zambrano-Luna:} Conceptualization, methodology, software, data curation, writing, visualization.
\textbf{Tianxu Wang:} Methodology, validation, formal analysis, writing.
\textbf{Hao Wang:} Conceptualization, supervision, validation, writing, funding acquisition.

\vspace{0.5cm}

\textbf{Competing Interests:}
There are no competing interests.
\end{flushleft}

\singlespacing
\begin{abstract}
\noindent Cyanobacterial blooms (CBs) pose significant global challenges due to their harmful toxins and socio-economic impacts, with nutrient availability playing a key role in their growth, as described by ecological stoichiometry (ES). However, real-world ecosystems exhibit spatial heterogeneity, limiting the applicability of simpler, spatially uniform models. To address this, we develop a spatially explicit partial differential equation model based on ES to study cyanobacteria in the epilimnion of freshwater systems. We establish the well-posedness of the model and perform a stability analysis, showing that it admits two linearly stable steady states, leading to either extinction or saturation. We use the finite elements method to numerically solve our system on a real lake domain derived from Geographic Information System (GIS) data and realistic wind conditions extrapolated from ERA5-Land. Our numerical results highlight the importance of lake shape and size in CB monitoring, while global sensitivity analysis using Sobol Indices identifies light attenuation and intensity as primary drivers of bloom variation, with water movement influencing early bloom stages and nutrient input becoming critical over time. This model supports continuous water-quality monitoring, informing agricultural, recreational, economic, and public health strategies for mitigating CBs.
\end{abstract}

\vspace{0.5cm}

\begin{flushleft}

    \vspace{0.5cm}
    
    \textbf{Acknowledgments:} This research project, especially J.S. and B.A.Z.L, was supported by a research grant (PI: H.W.) from Alberta Conservation Association. H.W.'s research was partially supported by an NSERC Individual Discovery Grant (RGPIN-2020-03911), an NSERC Discovery Accelerator Supplement Award (RGPAS-2020-00090), and a Senior Canada Research Chair. 
    \vspace{0.5cm}

    \textbf{Code Availability:} The code and data that support the findings of this study are available on GitHub with the identifier \href{https://github.com/B-A-Zambrano-Luna/Decoding-the-spatial-spread-of-cyanobacterial-blooms-in-an-epilimnion.git}{Modelling cyanobacteria population and nutrient stoichiometry}.
\end{flushleft}

\end{titlepage}

\singlespacing
\normalsize
\newpage

\section{Introduction}
Cyanobacteria, known for their distinct blue-green pigments, are a group of photosynthetic phytoplankton that play a central role in primary production and biogeochemical cycling in freshwater ecosystems (\cite{paerl2008blooms, wilhelm2011}). In recent decades, the global proliferation of cyanobacterial blooms (CBs) has raised serious concerns, as many bloom-forming species can produce toxins harmful to humans and wildlife (\cite{huisman2018cyanobacterial, kotak1996microcystin, Hardy2015}). CBs are associated with various economic and environmental problems, such as deterioration of water quality, restrictions on recreational activities, disruption of aquatic food webs, and being olfactorily unpleasant (\cite{paerl2009climate}). The severity and frequency of these blooms are influenced by changing climatic conditions such as rising temperatures, altered precipitation patterns, and anthropogenic eutrophication (the enrichment of water bodies with nutrients such as phosphorus and nitrogen) (\cite{paerl2008blooms, wilhelm2011}). 

Existing modelling efforts show that ecological stoichiometry provides a robust framework for understanding nutrient-driven population growth and species interactions (\cite{heggerud2020}). For cyanobacteria, stoichiometric constraints often center on intracellular phosphorus content and how this interacts with light, temperature, and other ecological parameters. Although several ordinary differential equation (ODE) models have shed light on stoichiometric constraints in plankton species (\cite{wang2007}), many freshwater ecosystems exhibit pronounced heterogeneity in space due to gradients in nutrient concentrations, light availability, and water movement. As a result, partial differential equations (PDEs) can more accurately capture the spatiotemporal patterns arising from organism movement and resource flow (\cite{Cantrell2004}). Specifically, incorporating diffusion, advection, and vertical or horizontal mixing processes can offer deeper insights into how blooms initiate and spread and how nutrient supply interacts with spatially varying environmental conditions (\cite{huisman1994, cao2006}). 

We propose and analyze a reaction-diffusion-advection model that incorporates ecological stoichiometry to study the interactions among cyanobacterial biomass, intracellular phosphorus, and dissolved mineral phosphorus. While previous studies \cite{hsu2010,hsu2014,Hsu2017} have focused on nutrient stoichiometry in a vertical domain, unstirred chemostat, our interest lies in understanding horizontal heterogeneity in large lakes, such as Pigeon Lake. By modeling these dynamics, we aim to better understand the movement of cyanobacteria, which can be observed through satellite imagery. To achieve this, we introduce a system of reaction-diffusion-advection equations that account for horizontal diffusion and advection in the lake’s epilimnion, assuming a constant diffusion rate in the horizontal domain.

Our work focuses on three major components. First, we rigorously formulate a mechanistic model and establish the well-posedness of the system in a bounded domain by proving the global existence, uniqueness, and boundedness of solutions under biologically relevant assumptions. Second, we investigate the long-term behaviour of the model through rigorous stability analysis of equilibrium states, identifying conditions under which cyanobacteria persist or die out. To facilitate the above analyses, we consider two equivalent systems, (\ref{model:main}) and (\ref{eq: BQP}), to handle singularities and blow-up problems. Finally, we perform extensive numerical simulations to explore spatial dynamics in one- and two-dimensional spatial domains. We employ finite-difference schemes to illustrate key wavefront and boundary-layer phenomena in the one-dimensional setting. For the two-dimensional case, we utilize the finite element method on a lake-shaped domain to improve realism, where spatial heterogeneity and irregular boundary geometry can strongly influence bloom formation and nutrient distribution. Both one- and two-dimensional simulations incorporate a function interpolated to real wind data from Pigeon Lake, Alberta, Canada, in 2023.

Given that parameter uncertainty is often significant in real freshwater systems, we perform a global sensitivity analysis to identify which physical parameters, such as epilimnion depth, water exchange rate, and background light attenuation, are most influential on cyanobacteria dynamics. Our results reveal how spatially heterogeneous resource distributions, wind-driven advection, physical lake conditions, and nutrient dynamics shape the persistence or collapse of a cyanobacteria population. We discuss how these findings have implications for water resource management, providing guidance on which parameter regimes may be most critical to monitor or control in the context of bloom prevention and mitigation efforts.

\section{Model formulation}
\label{sec:modelformulation}
Numerous models exist that investigate the qualitative dynamics of cyanobacteria in a spatially homogeneous environment. One such model given by \cite{heggerud2020} proposed the following stoichiometric equation for the growth of cyanobacteria,
\begin{align}
\label{model: ODE Heggerud}
    \begin{aligned}
        \frac{dB}{dt} &= r B\left(1-\frac{Q_m}{Q}\right)h(B) - lB - \frac{D}{z_m}B, \\
    \end{aligned}
\end{align} 
where $B$ and $Q$ describe the concentration of carbon biomass of cyanobacteria and intracellular phosphorus-to-carbon cell quota $(p/B)$, respectively. Both nutrients and light limit the logistic growth of cyanobacteria. The empirically validated Droop form models the former and the latter by the ubiquitous Monod form. The difference in mathematical expression stems from the two forms of energy being absorbed and processed differently (\cite{wang2022}). Adhering to the Lambert-Beer law (\cite{huisman1994}), with $r$ representing the growth rate, the light-dependent cyanobacterial growth function is 
\begin{align}
\label{eq: h}
    h(B) = \frac{1}{z_m(kB+K_{bg})}\ln\left(\frac{H+I_{in}}{H+I(z_m,B)}\right),
\end{align}
where
\begin{align}
\label{eq: I}
    I(s,B) = I_{\text{in}}\text{exp}[-(K_{\text{bg}}+kB)s]
\end{align}
describes the light intensity given at a water depth $s$. The parameters $I_{\text{in}}$, $K_{\text{bg}}$, and $k$ are described in Table \ref{tab:parameters}. Cyanobacteria uptake nutrients from the environment and lose nutrients through dilution due to growth. The nutrient uptake function is given by
\begin{align}
\label{eq: rho}
    \rho(Q,P) = \rho_m\left(\frac{Q_M-Q}{Q_M-Q_m}\right)\frac{P}{P+M},
\end{align}
where $\rho_m$, $M$, and $P$ are the maximum phosphorus uptake rate, the half-saturation coefficient for cyanobacteria phosphorus uptake, and the concentration of mineral phosphorus in the epilimnion, respectively. The parameter $l$ represents the loss rate of cyanobacteria due to respiration and other factors. The other parameters are given in Table \ref{tab:parameters}.
 
Existing modelling efforts incorporating ecological stoichiometry are commonly explored using ODE frameworks. However, empirical evidence suggests that environments' spatial scale and structure can significantly influence population interactions (\cite{Cantrell2004}). Thus, extending the existing ODE model to apply to a spatially heterogeneous environment is natural. In this study, we mechanistically formulate a reaction-diffusion-advection PDE model to investigate the dynamics of cyanobacteria in a spatially heterogeneous lake. To provide an accurate and versatile model that can be fitted to various data exports, we take a top-down view of a lake and consider it within the bounded domain $\Omega \subset \mathbb{R}^2$. We consider a lake with two layers: the epilimnion and the hypolimnion. The epilimnion, with an assumed depth of $z_m$, is where cyanobacteria predominantly reside. Below the epilimnion lies the hypolimnion, a deeper layer typically more stable and less affected by surface conditions. The stratification between these two layers impacts the transport and mixing of nutrients and the movement of cyanobacteria. Since cyanobacteria generally do not sink unless they are deceased (\cite{zhao2014, cao2006}), we assume their population exists idly (without external force) in the epilimnion layer.

Cyanobacteria and nutrient phosphorus transport are mainly affected by three types of movements. They follow random movement near the lake's surface caused by water diffusion, with diffusion rates $\alpha$ and $\beta$. Within the epilimnion layer, their directional movement is determined by water affected by wind (\cite{cao2006, zhang2021wind}), with a vector $\Vec{v} = (u(t),v(t)) \in \mathbb{R}^2$ describing the wind movement. We consider water exchange between the epilimnion and hypolimnion, with a water exchange rate of $D$. The nutrients and cyanobacteria will sink or become buoyant due to the force of this water exchange. Following \cite{heggerud2020}, the water exchange rate is inversely proportional to the depth of the epilimnion. On the other hand, since the hypolimnion in a lake is significant and stable enough, the phosphorus variation in the hypolimnion can be considered negligible, and we assume the phosphorus concentration is constant at $P_h$. Hence, $\frac{D}{z_m}\left(P_h - P\right)$ indicates the vertical exchange through the thermocline from the hypolimnion to the epilimnion. If $P_h > P$, there is an extra phosphorus input from the hypolimnion to the epilimnion; otherwise, the phosphorus concentration will be diluted. Since $Q$ is not a physical quantity but a ratio, it is more reasonable to track internal ($p$) and external ($P$) phosphorus as spatially heterogeneous quantities within the model.

\begin{table}[ht]
\centering
\caption{Parameters and their respective meanings and units used in the model.}
\resizebox{0.8\textwidth}{!}{%
\begin{tabular}{llr}
Variable & Meaning & Unit \\
\midrule
$\alpha$ & cyanobacterial diffusion coefficient & $\mathrm{m^2/\text{day}}$ \\
$\beta$ & mineral phosphorus diffusion coefficient & $\mathrm{m^2/\text{day}}$ \\
$\beta_B$ & cyanobacterial advection scalar & $\mathrm{unitless}$ \\
$\beta_P$ & mineral phosphorus advection scalar & $\mathrm{unitless}$ \\
$r$ & maximum cyanobacterial production rate & $1/\text{day}$ \\
$Q_m$ & cyanobacterial cell quota at which growth ceases & $\mathrm{mgP/mgC}$ \\
$Q_M$ & cyanobacterial cell quota at which nutrient uptake ceases & $\mathrm{mgP/mgC}$ \\
$z_m$ & depth of the lake epilimnion & $\mathrm{m}$ \\
$k$ & cyanobacterial-specific light attenuation & $\mathrm{m^2/mgC}$ \\
$K_{bg}$ & background light attenuation & $1/\mathrm{m}$ \\
$H$ & Half-saturation coefficient of light-dependent cyanobacterial production & $\mathrm{\mu mol/(m^2 \cdot \text{day})}$ \\
$I_{in}$ & light intensity at the surface of the water & $\mathrm{\mu mol/(m^2 \cdot \text{day})}$ \\
$l$ & loss of cyanobacterial due to respiration  & $1/\text{day}$ \\
$D$ & water exchange between stratified lake layers & $\mathrm{m/\text{day}}$ \\
$\rho_m$ & Maximum cyanobacterial phosphorus uptake rate & $\mathrm{mgP/mgC/day}$ \\
$P_h$ & dissolved mineral phosphorus concentration in hypolimnion & $\mathrm{mgP/m^2}$\\
$M$ & Half-saturation coefficient for cyanobacterial nutrient uptake & $\mathrm{mgP/m^2}$ \\
\bottomrule
\end{tabular}%
}
\label{tab:parameters}
\end{table}

Based on the above discussion, the full model on $\Omega \times (0,\infty)$ is then given by
\begin{align}\tag{$\mathcal{S}_p$}
    \left\{
    \begin{aligned}
        \frac{\partial B(x,t)}{\partial t} &= \underbrace{\alpha \Delta B}_{\text{diffusion}} - \underbrace{\beta_B \Vec{v}(t) \nabla B}_{\text{advection}} + \underbrace{r \left(1-Q_m\frac{ B}{p}\right)h(B) B}_{\text{limited growth term}} - \underbrace{l B}_{\text{loss}} - \underbrace{\frac{D}{z_m} B}_{\text{vertical exchange}}, \\
        \frac{\partial p(x,t)}{\partial t} &= \underbrace{\alpha \Delta p}_{\text{diffusion}}- \underbrace{\beta_B \Vec{v}(t) \nabla p}_{\text{advection}}+\underbrace{\eta(B,p,P)}_{\text{uptake by bacteria}} - \underbrace{l p}_{\text{phosphorus loss}} - \underbrace{\frac{D}{z_m}p}_{\text{vertical exchange}}, \\
        \frac{\partial P(x,t)}{\partial t} &= \underbrace{\beta \Delta P}_{\text{diffusion}} - \underbrace{\beta_P \Vec{v}(t) \nabla P}_{\text{advection}} + \underbrace{\frac{D}{z_m}(P_h - P)}_{\text{vertical exchange}} - \underbrace{\eta(B,p,P)}_{\text{uptake by bacteria}} + \underbrace{l p}_{\text{phosphorus recycling}}, 
        \hspace{10pt}
    \end{aligned}
    \label{model:main}
    \right.
\end{align}
with boundary and initial conditions
\begin{equation}
    \label{eq: bc ic}
    \begin{aligned}
        &\partial_\nu B(x,t) = \partial_\nu p(x,t) = \partial_\nu P(x,t) = 0, \qquad (x,t) \in \partial\Omega \times (0, \infty), \\
        &(B(x,0), p(x,0), P(x,0)) = (B_0(x), p_0(x), P_0(x)), \qquad x \in \Omega.
    \end{aligned}
\end{equation}
Here, $B$ ($\mathrm{mgC/m^2}$) represents the concentration of cyanobacterial biomass, and 
\begin{equation*}
    \eta(B,p,P) = \rho_m \left(\frac{Q_MB-p}{Q_M-Q_m} \right)\frac{P}{P+M}.
\end{equation*}
Using this function, we remove the singularity of the original function $\rho(B,p,P)$.
The remaining parameters are provided in Table \ref{tab:parameters}.


\section{Dynamical Analysis}
\label{sec:analysis}

\subsection{Assumptions}

\begin{itemize}
    \item[($A_I$)] The initial conditions $B_0$, $p_0$, $P_0$ satisfy 
    \begin{equation*}
         (B_0,p_0, P_0) \in [W^{2,\infty}(\Omega)]^3,\text{ }   B_0, p_0, P_0 > 0, \text{ and } Q_m\le \frac{p_0}{B_0}\le Q_M.
    \end{equation*}

    \item[($A_C$)] All parameters in the model are positive constants.

\end{itemize}

\subsection{Well-posedness}

For biological rationality, it is necessary to show the well-posedness of the solutions to the model. Therefore, we devote this section to studying the existence and uniqueness of model (\ref{model:main}). We introduce the following lemma to resolve the challenge from the unboundedness of the reaction term in the first equation.

\begin{lemma}\label{lem:Qbound}
Assume that $(A_I)$, $(A_C)$, and $(B,p,P)\in C^{2, 1} \big(\overline{\Omega} \times (0, T_{\max}); \mathbb R_{+}^3 \big) \cap C \big( \overline{\Omega} \times [0, T_{\max}); \mathbb R_{+}^3 \big)$ is a solution of the system \eqref{model:main} -- \eqref{eq: bc ic}. Then it satisfies
 \begin{equation}\label{ineq:quota}
 Q_m\le \frac{p}{B}\le Q_M \,\text{ in 
 }\,\overline{\Omega}\times [0,T], \; T< T_{\max}. 
 \end{equation}
\end{lemma}
\begin{proof}
We first prove the $Q_m\le \frac{p}{B}$.
Let $y = p-Q_mB$ and $y(\cdot,0)= p_0-Q_mB_0\ge 0$ by $(A_I)$. Then it follows that
\begin{equation*}
\begin{aligned}
&\quad \ \partial_t y -\alpha\Delta y + \beta_B \Vec{v}(t) \nabla y  +ly + \frac{D}{z_m} y - \rho(p/B,P)B +  \frac{Q_m r}{p} h(B) B y = 0.
\end{aligned}
\end{equation*} 
Define an operator $\mathcal{L}_1: C^{2,1}\left(\bar{\Omega}\times(0,T_{\max});\mathbb{R}_+^3\right)\cap C\left(\bar{\Omega}\times[0,T_{\max});\mathbb{R}_+^3\right)\to C\left(\bar{\Omega}\times[0,T_{\max});\mathbb{R}_+^3\right)$,
\begin{equation*}
\begin{aligned}
\mathcal{L}_1(y)&:=\ \partial_t y -\alpha\Delta y + \beta_B \Vec{v}(t) \nabla y  +ly + \frac{D}{z_m} y - \rho(p/B,P)B +  \frac{Q_m r}{p} h(B) B y\\
&=\partial_t y -\alpha\Delta y + \beta_B \Vec{v}(t) \nabla y  +ly + \frac{D}{z_m} y + \frac{\rho_m }{Q_M-Q_m}\frac{P}{P+M} \left(y - (Q_M- Q_m) B\right) \\
& \quad + \frac{Q_m r}{p} h(B) B y .
\end{aligned}
\end{equation*}
It's easy to check that, 
\begin{equation*}
\mathcal{L}_1(0) = -\rho_m \frac{P}{P+M}  B \leq 0,
\end{equation*} for each $(x,t)\in \overline{\Omega}\times [0,T_{\max})$. The standard comparison principle yields $y\ge 0$ in $\overline{\Omega}\times [0,T_{\max})$. Likewise, we can use a similar argument to prove the second inequality of \eqref{ineq:quota}.
\end{proof}

\begin{remark}
    Lemma \ref{lem:Qbound} also holds for either $B$ and $p$ approaching 0.
\end{remark}

\begin{lemma}[Local existence]
    \label{lem: local existence}
    The system \eqref{model:main} -- \eqref{eq: bc ic} has a unique maximal solution $U = (B, p, P)$ satisfying
    \begin{equation}
        \label{eq: smoothness}
        U \in C^{2, 1} \big(\overline{\Omega} \times (0, T_{\max}); \mathbb R_{\geq0}^3 \big) \cap C \big( \overline{\Omega} \times [0, T_{\max}); \mathbb R_{+}^3 \big),
    \end{equation}
    where $T_{\max} \in (0, \infty]$. If $T_{\max} < \infty$, then
    \begin{equation}
        \label{eq: blow up condition}
        \lim_{t \to T_{\max}} \big( ||B (\cdot, t)||_{L^\infty} + ||p (\cdot, t)||_{L^\infty} + ||P (\cdot, t)||_{L^\infty} \big) = \infty.
    \end{equation} 
\end{lemma}
\begin{proof}
We consider auxiliary system given by
\begin{equation*}
    \partial_t U =  \mathcal \nabla (\mathbb{A} \nabla U) + \Phi (U, \nabla U), \ (x,t) \in \Omega \times (0, \infty), 
\end{equation*}
where $U=(B,p,P)$, $\mathbb{A} = \text{diag}\{\alpha, \alpha, \beta\}$
and 
\begin{equation}
    \label{eq: vector Phi}
    \Phi (U,\nabla U) = 
    \begin{pmatrix}
        - \beta_B \Vec{v}(t) \nabla B + r\left(1-Q_m\frac{B}{p}\right)h(B) B - l B -  \frac{D}{z_m} B \\
    - \beta_B \Vec{v}(t) \nabla p+\eta\left(B,p,P\right) B - l p -  \frac{D}{z_m}p \\
    - \beta_P \Vec{v}(t) \nabla P +  \frac{D}{z_m}(P_h - P) + P_{in} - \eta\left(B,p,P\right) B + l p
    \end{pmatrix}.
\end{equation}
    Fix $p > n$, $\epsilon > 0$ and define
    \begin{equation*}
        V = \{ v \in W^{1, p} (\Omega; \mathbb{R}^3): v(x) \in G = (0, \infty)^3, \ \forall x \in \overline{\Omega} \}.
    \end{equation*}
    Since $\mathbb{A} = \text{diag}\{\alpha, \alpha, \beta\}$ and $\alpha$, $\beta$ are real positive numbers, then all its eigenvalues have positive real parts. Moreover, assumption ($A_I$) implies that the initial values belong to $V$. Therefore, the local existence of the solution $U$ is guaranteed by (\cite{amann1990}, p. 17), where $U$ satisfies
    \begin{equation*}
        U \in C \big( [0, T_{\max}), V \big) \cap C^{2, 1} \big( \overline{\Omega} \times (0, T_{\max}), \mathbb R^3 \big)
    \end{equation*}
    and $T_{\max} \in (0, \infty]$ is the lifespan defined by
    \begin{equation*}
        T_{\max} := \sup \{ T > 0: U(\cdot, t) \in V, \ \forall t \in [0, T] \}.
    \end{equation*}
    By the Sobolev Embedding Theorem and the fact that $U(\cdot, t) \in C(\overline{\Omega}; \mathbb{R}^3)$ for all $t \geq 0$, we obtain
    \begin{equation*}
        U \in C \big( [0, T_{\max}), C(\overline{\Omega}; \mathbb{R}^3) \big) = C \big( \overline{\Omega} \times [0, T_{\max}); \mathbb R^3 \big).
    \end{equation*}
    Since \eqref{model:main} is uniformly positive definite, together with assumption ($A_I$), the solution to equation \eqref{model:main} -- \eqref{eq: bc ic} satisfies (Lemma 3.1 \cite{wang2025existence})
    \begin{equation*}
        B(x,t), p(x,t), P(x,t) > 0, \ \forall (x,t) \in \overline{\Omega} \times (0, T_{\text{max}}).
    \end{equation*}
    Then, \eqref{eq: smoothness} holds. In order to exclude the possibility that $U$ approaches to $\partial V$ as $t\to T_{\max}$ if $T_{\max}<\infty$, we prove $U$ stays away from $\partial V$. 
    Assume $T_{\max}<\infty$.     
    Consider the auxiliary problem
    \begin{equation}\label{eqn:aux}
    \left\{\begin{aligned}
    \begin{alignedat}{2}
        &\partial_t \tilde B =\alpha\Delta \tilde B - \beta_B \vec{v}(t) \nabla \tilde B  -l\tilde B -  \frac{D}{z_m} \tilde B && \,\,\,\,\text{ in } \Omega\times(0,\tilde T_{\max})\\
    &\partial_\nu \tilde B = 0 && \,\,\,\,\text{ in } \partial\Omega\times    (0,\tilde T_{\max}) \\
    &\tilde B(x,0) = B_0(x) && \,\,\,\,\text{ on }\overline{\Omega}. 
    \end{alignedat}
    \end{aligned}\right.
    \end{equation} By \cite{amann1990} and $B_0>0$ in $\overline{\Omega}$, there exists a unique nontrivial solution $\tilde B\in C^{2, 1} \big(\overline{\Omega} \times (0, T_{\max}) \big) \cap C \big( \overline{\Omega} \times [0, T_{\max}) \big)$ of \eqref{eqn:aux}. One can easily show that $0\le \tilde B\le \|B_0\|_{L^\infty(\Omega)}$ in $\overline{\Omega}\times [0,\tilde T_{\max})$, which gives $\tilde T_{\max}=\infty$. 
    Let $y = \tilde B -B $. Then $y$
    satisfies 
    \begin{equation*}
    \left\{\begin{aligned}
    \begin{alignedat}{2}
        &\partial_t  y =\alpha\Delta  y - \beta_B \vec{v}(t) \nabla  y  -ly -  \frac{D}{z_m}y -r\left(1-Q_m\frac{B}{p}\right)h(B) B && \,\,\,\,\text{ in } \Omega\times(0,T_{\max})\\
    &\partial_\nu y = 0 && \,\,\,\,\text{ in } \partial\Omega\times    (0,T_{\max}) \\
    & y(x,0) = 0 && \,\,\,\,\text{ on }\overline{\Omega}. 
    \end{alignedat}
    \end{aligned}\right.
    \end{equation*} Due to Lemma \ref{lem:Qbound}, we have $\partial_t  y -\alpha\Delta  y + \beta_B \vec{v}(t) \nabla  y  +(l +  \frac{D}{z_m})y \le 0$ in $\Omega\times(0,T_{\max})$. The maximum principle shows that $y\le 0$ in $\overline{\Omega}\times[0,T_{\max})$, hence $\tilde B\le B$ in $\overline{\Omega}\times[0,T_{\max})$. Harnack's inequality for parabolic equations (see \cite{evans2010}) ensures that for each $0<t_1<t_2$ and $K\subset\subset\Omega$, there exists a positive constant $C$ such that 
    \begin{align*}
        \sup_{x\in K}\tilde B(x,t_1)\le C\inf_{x\in K}\tilde B(x,t_2).
    \end{align*} For each $x\in \Omega$, we can choose $K$ and $t<T_{\max}$ such that $ \sup_{x\in K}\tilde B(x,t)>0$.
    Then we have $ 0<\sup_{x\in K}\tilde B(x,t)<C\inf_{x\in K}\tilde B(x,T_{\max})$, which implies $\tilde B(x,T_{\max})>0$. If $x\in \partial \Omega$ and $\tilde B(x,t) =0$, Hopf's lemma yields a contradiction to the fact that $\partial_\nu \tilde B =0$ on $\partial\Omega$. Therefore, $\tilde B(x,t)>0$ over $\overline{\Omega}\times [0,T_{\max}]$. In conclusion, for each $x\in \overline{\Omega}$
    \begin{align*}
        \liminf\limits_{t\to T_{\max}}B(x,t)\ge \tilde B(x,T_{\max})>0.
    \end{align*}
    By using similar arguments, we can show $p$ and $P$ do not approach zero as $t\to T_{\max}$, so that the solution $U=(B,p,P)$ stays away from $\partial V$, which concludes that \eqref{eq: blow up condition} holds by \cite{amann1990}. 
\end{proof}

\subsection{\texorpdfstring{$L^\infty$}{L-infinity} boundedness}

\begin{lemma}
    \label{lem: B bounded}
    Assume that $(A_I)$ and $(A_C)$.
    Then there exists a constant $\bar B$ such that the solution to equation \eqref{model:main} -- \eqref{eq: bc ic} satisfies 
    \begin{equation}
        \label{eq: B bounded}
        B(x,t)\leq \max\{\|B_0\|_{L^\infty(\overline\Omega)},\bar{B}\}, \quad \text{ for all } (x,t) \in \overline{\Omega} \times [0, T_{\max}).
    \end{equation}
\end{lemma}

\begin{proof}
From equation \eqref{eq: h}, \eqref{eq: I}, and Lemma \ref{lem:Qbound}, we obtain
\begin{equation*}
    h(B) \leq \frac{1}{z_m (k B + K_{bg})}\Big(\frac{Q_M-Q_m}{Q_M}\Big)\ln \left(\frac{H + I_{\text{in}}}{H}\right).
\end{equation*}
    It follows from the positivity of $B$ and $p$, the first equation of \eqref{model:main} imply that
    \begin{equation}
        \label{eq: proof of B bounded}
        \partial_t B - \alpha \Delta B + \beta_B \vec{v}(t) \nabla B \leq  f(B)\quad \text{ in } \Omega\times (0,T_{\max}),
    \end{equation}  where 
    \[f(B)=\left( \frac{r}{z_m (k B + K_{bg})}\Big(\frac{Q_M-Q_m}{Q_M}\Big)\ln \left(\frac{H + I_{\text{in}}}{H}\right) - l -  \frac{D}{z_m}\right)B.\] 
    Let us denote by a nonzero constant $\bar B $ such that $f\big(\bar{B}\big)=0.$
     Then either $\bar B< 0$ or $\bar B>0$ for different parameters of $f$. If $\bar B< 0$, $f(B)<0$ for all $(x,t)\in \Omega\times (0,T_{\max})$. Then the maximum principle and $(A_I)$ show that $B\le \|B_0\|_{L^\infty}$. (only if $\partial \Omega \in C^2$.)
    If $\bar B >0$, notice that $f(B)\le 0$ for all $B\ge\bar B$.
    Let us define an operator $\mathcal{L}_2: C^{2,1}\left(\bar{\Omega} \times (0,T_{\max}) \right)\cap C\left(\bar{\Omega} \times [0,T_{\max})\right) \to C\left(\bar{\Omega} \times [0,T_{\max})\right)$, with
    \begin{equation*}
        \begin{aligned}
        \mathcal{L}_2(B) &= \partial_t B-\alpha \Delta B+\beta_B \vec{v}(t) \nabla B - r\left(1-\frac{Q_mB}{p}\right)h(B)B -lB-\frac{D}{z_m}B\\
        &\ge \partial_t B - \alpha \Delta B + \beta_B \vec{v}(t) \nabla B -  f(B).
        \end{aligned}
    \end{equation*} 
    Then $\mathcal{L}_2(C)\ge 0$ for  $C=\max\{\|B_0\|_{L^\infty(\overline{\Omega})},\bar B\}$. Since $B_0\le C$ in $\Omega $, the standard comparison principle concludes our lemma.

\end{proof}
\begin{remark}
    Notice that $\bar B$ depends on the parameters in reaction terms. Indeed, direct calculation gives
    \[ \bar B = \frac{1}{k}\left(\frac{r}{z_ml+D}\Big(\frac{Q_M-Q_m}{Q_M}\Big)\left(\frac{H + I_{\text{in}}}{H}\right)-K_{bg}\right)\]
    If $\bar B>0$ and $B_0$ is so small that $B_0< \bar B$, then the uniform bound $\bar B$ of $B$ depends on the parameters above. Biological speaking, cyanobacteria blooms depending on environmental, internal factors and its ability to absorb the light. Note that the  bound of cyanobacteria is proportional to the difference $Q_M-Q_m$.
\end{remark}

\begin{lemma}
    \label{lem: p bounded} Assume that $(A_I)$ and $(A_C)$.
    There exists a constant $\bar{p}>0$ such that the solution to equation \eqref{model:main} -- \eqref{eq: bc ic} satisfies 
    \begin{equation}
        \label{eq: p bounded}
        p(x,t)\le \bar p \quad \text{ for all }(x,t) \in \overline{\Omega} \times [0, T_{\max}).
    \end{equation}
\end{lemma}

\begin{proof}
Lemma \ref{lem: B bounded} and lemma \ref{lem:Qbound} conclude the corollary by letting $\bar p = Q_M\max\{\|B_0\|_{L^\infty}, \bar B\}$.
\end{proof}

\begin{lemma}
    \label{lem: P bounded} Assume that $(A_I)$, $(A_C)$, and $P_{\text{in}}$ is bounded.
    There exists a constant $\bar{P}>0$ such that the solution to equation \eqref{model:main} -- \eqref{eq: bc ic} satisfies 
    \begin{equation}
        \label{eq: P bounded}
        P(x,t)\leq \max\{\|P_0\|_{L^\infty(\Omega)},\bar{P}\}, \quad \text{ for all } (x,t) \in \overline{\Omega} \times [0, T_{\max}).
    \end{equation}
\end{lemma}

\begin{proof}
From equation \eqref{eq: rho}, along with the boundedness of \( B \), the second equation in \eqref{model:main} implies that
    \begin{equation}
        \label{eq: proof of P bounded}
        \begin{aligned}
        \partial_t P - \beta \Delta P + \beta_P \vec{v}(t) \nabla P & =   \frac{D}{z_m}(P_h - P) + P_\text{in} - \rho(p/B,P)B+ lp\\
        \end{aligned}
    \end{equation}
    Let us denote by an operator $\mathcal{L}_3: C^{2,1}\left(\bar{\Omega} \times (0,T_{\max}) \right)\cap C\left(\bar{\Omega} \times [0,T_{\max})\right) \to C\left(\bar{\Omega} \times [0,T_{\max})\right)$, with
    \begin{equation*}
    \begin{aligned}
        \mathcal{L}_3(P) &= \partial_t P - \beta \Delta P + \beta_P \vec{v}(t) \nabla P  -   \frac{D}{z_m}(P_h - P) - P_\text{in} + \rho(p/B,P)B- lp\\
        &\ge - \frac{D}{z_m}(P_h - P) - P_\text{in} + \frac{\rho_m\bar p}{Q_m}\frac{P}{P+M} - l\bar p=: g(P).
        \end{aligned}
    \end{equation*}
Let $\bar{P}$ be a positive zero of $g$.
    Note that $g(P)\ge 0$ for all $P\ge \bar P$. Since $\mathcal{L}_3(C)\ge 0$ for $C=\max\{\|P_0\|_{L^\infty(\overline{\Omega})},\bar P\}$ and $P_0\le C$,
    the comparison principle shows that \eqref{eq: P bounded} holds.
\end{proof}

Lemma \ref{lem:Qbound}, \ref{lem: B bounded}, \ref{lem: P bounded}, together with the criterion from Lemma \ref{lem: local existence}, we have the following Theorem.
\begin{theorem}[Global boundedness]
    \label{th: global existence}
    Let $\Omega\subset \mathbb{R}^n$ $(n\geq 1)$ be a bounded domain with a smooth boundary. Under Assumptions $(A_I)$ -- $(A_c)$, equation \eqref{model:main} -- \eqref{eq: bc ic} has a unique maximal solution $U = (B, p, P) \in C \big( \overline{\Omega} \times [0, \infty); \mathbb R_{\geq 0}^3 \big) \cap C^{2, 1} \big( \overline{\Omega} \times (0, \infty); \mathbb R_+^3 \big)$. Furthermore, there exists a constant $C>0$ independent of t such that 
    \begin{equation*}
        \|B(\cdot, t)\|_{L^{\infty}} + \|p(\cdot, t)\|_{L^{\infty}} + \|P(\cdot, t)\|_{L^\infty} \leq C, \qquad \forall t \in \left[0, \infty \right).
    \end{equation*}
\end{theorem}

\subsection{Linear stability analysis around equilibria with small movements}
In natural environments, wind plays a significant role in shaping the spatial distribution of cyanobacteria by influencing their transport through air convection. Given that cyanobacteria exhibit minimal random movement, it was initially hypothesized that spatial heterogeneity in their distribution could arise due to wind-driven transport. To explore this, we conducted a local stability analysis to assess whether small spatial perturbations around homogeneous steady states could grow over time, potentially leading to pattern formation. Our analysis considered a simplified scenario where the wind vector was assumed to be constant, denoted as $\vec{v}(t) = \mathbf{v} = (v_1, v_2) \in \mathbb{R}^2$. Additionally, we assumed that the horizontal transport of cyanobacteria and phosphorus through diffusion and advection was relatively minor compared to their vertical movement, leading to small diffusivity parameters ,  and a small wind-driven transport magnitude $|\mathbf{v}|$ . Under these assumptions, we employed first-order perturbation theory to approximate the eigenvalues of the linearized system and determine the stability of the homogeneous steady states.

Our local stability analysis revealed that weak diffusion and wind transport alone are insufficient to sustain long-term spatial heterogeneity. To verify this, we conducted numerical simulations incorporating real wind data. See Figure~\ref{fig:1D_Stability_no_ph} and \ref{fig:1D_Time_Series_Real_Wind}. The results showed that while spatial heterogeneity temporarily emerged during the early stages of cyanobacterial bloom, it did not persist over time. Instead, cyanobacteria eventually reached a saturated and spatially uniform state, even under realistic wind conditions. This suggests that the spatial heterogeneity observed in cyanobacterial blooms is primarily a transient phenomenon that occurs only in the initial stages of bloom formation. These findings highlight the importance of temporal dynamics in understanding bloom patterns. While wind-driven transport may influence the early development of spatial heterogeneity, additional mechanisms—such as biological interactions, nutrient availability, or external environmental disturbances—may be required to sustain long-term spatial pattern formation. This motivates further investigation into factors beyond wind-driven transport that could contribute to persistent spatial structuring in cyanobacterial blooms such as lake shape and size.

\subsubsection{Equivalent systems}
We first examine an equivalent PDE system that shares the same fundamental dynamics as our original model to analyze the stability of steady states without diffusion and wind advection. This equivalent system possesses a nonsingular structure, allowing us to investigate the corresponding ODE system's stability properties systematically. By leveraging this formulation, we establish a rigorous connection between the PDE and its reduced ODE model. This ensures that our stability analysis remains well-posed and interpretable within the broader dynamical framework.

Let $U=(B,p,P)$ be the solution of \eqref{model:main} -- \eqref{eq: bc ic} with the assumptions $(A_I)$ and $(A_C)$. We introduce a new equation for the variable $Q=\frac{p}{B}$ to obtain
\begin{align*}
    \frac{\partial Q}{\partial t} = \frac{\partial (p/B)}{\partial t} &= \alpha\left(\frac{\Delta p}{B}-\frac{p}{B^2}\Delta B\right) - \beta_B\mathbf{v}\left(\frac{\nabla p}{B}-\frac{p}{B^2}\nabla B\right)+\frac{\eta(B,p,P)}{B}-rQ\left(1-\frac{Q_m}{Q}\right)h(B) \\
    &= \alpha B^{-2}\nabla \cdot (B^2\nabla Q) - \beta_B\mathbf{v}\nabla Q +\rho(Q,P) - rQ\left(1-\frac{Q_m}{Q}\right)h(B) \\
    &= \alpha \Delta Q + \left(2\alpha\nabla \ln B-\beta_B\mathbf{v}\right)\cdot\nabla Q+\rho(Q,P) - rQ\left(1-\frac{Q_m}{Q}\right)h(B).
\end{align*} 
$Q$ represents the phosphorus ratio in cyanobacterial cells.
The advection term \( 2\alpha\nabla (\ln B) \cdot \nabla Q \) represents the movement of \( Q \) influenced by \( B \). This indicates that \( Q \) not only diffuses but also moves in the opposite direction of the gradient of \( \ln B \) as \( B \) evolves. If \( B \) is spatially constant, the advection term vanishes, implying that when \( B \) becomes spatially homogeneous, \( Q \) will eventually stabilize to a constant value. Otherwise, \( Q \) remains non-constant. Since $Q$ is a cell-level property, its diffusion should depend on the density of cyanobacteria ($B$) in the surrounding area. Specifically, from the diffusion-like term  $B^{-2} \nabla \cdot (B^2 \nabla Q) $, the balance between $B^2$ and $B^{-2}$ in the diffusion equation ensures that the spread of the cell quota $Q$ is stable across varying cyanobacteria densities ($B$). In regions of high $B$, $B^2$ increases the flux, but $B^{-2}$ dampens the overall effect, preventing overly rapid diffusion. Conversely, in regions of low $B$, $B^2$ reduces the flux, but $B^{-2}$ amplifies the diffusion, ensuring $Q$ spreads effectively even in sparse populations. This interplay creates a self-regulating mechanism that balances diffusion, reflecting the biological need for stable resource distribution across different densities. 

In addition to the direct derivation of the equation $Q$ from the equations of $B$ and $p$, we derive the same diffusion-like term for $Q$ when $B$ diffuses on the heterogeneous environment over discretized space and time domain in Appendix~\ref{ap: deriveQdiffusion}. We observe the same diffusive formulation balanced by the distribution $B$ spreading $Q$. With the equation of the new variable $Q$, \eqref{model:main} is transformed to a system
\begin{equation}\label{eq: BQP}\tag{$\mathcal{S}_Q$}
    \left\{
    \begin{aligned}
        \frac{\partial B(x,t)}{\partial t} &= \alpha \Delta B - \beta_B \Vec{v}(t) \nabla B + r \left(1-Q_m\frac{ B}{p}\right)h(B) B - l B - \frac{D}{z_m} B, \\
        \frac{\partial Q(x,t)}{\partial t} &= \alpha B^{-2}\nabla\cdot(B^2 \nabla Q)- \beta_B \Vec{v}(t) \nabla Q-r(Q-Q_m)h(B)+\rho(Q,P), \\
        \frac{\partial P(x,t)}{\partial t} &= \beta \Delta P - \beta_P \Vec{v}(t) \nabla P + \frac{D}{z_m}(P_h - P) - \rho(B,Q)B + lBQ, 
        \hspace{10pt}
    \end{aligned}
    \right.
\end{equation}
with boundary and initial conditions
\begin{equation*}
    \begin{aligned}
        &\partial_\nu B(x,t) = \partial_\nu Q(x,t) = \partial_\nu P(x,t) = 0, \qquad (x,t) \in \partial\Omega \times (0, \infty), \\
        &(B(x,0), Q(x,0), P(x,0)) = (B_0(x), Q_0(x), P_0(x)), \qquad x \in \Omega,
    \end{aligned}
\end{equation*} where we assume the positivity of initial data. In addition, $Q_0=p_0/B_0$ and we have used the fact that $\nabla Q=-\frac{p}{B^2}\nabla B + \frac{1}{B}\nabla p$ for the boundary condition.  The existence and uniqueness of the solution to \eqref{eq: BQP} is guaranteed by the result of the system \eqref{model:main}. Furthermore, Lemma~\ref{lem:Qbound} ensures that \( Q \) remains bounded within \( [Q_m, Q_M] \). Two systems \eqref{model:main} and \eqref{eq: BQP} share the equilibrium points. We will prioritize using the system that provides computational or analytical advantages. The specific context and objectives of the analysis will guide this choice. For example, we use \eqref{eq: BQP} in the case of $B\to0$ and $p\to 0$, but $Q\to \hat{Q}$ for some $\hat Q>0$. In addition, if we need to consider a linearized problem near homogeneous equilibrium, we use \eqref{model:main} because of the difficulty from the nonlinearity of $B^{-2}\nabla\cdot(B^2\nabla Q)$.

\subsubsection{Steady-states without movement}

In the case of the homogeneous steady-state, we investigate the steady-state problem of \eqref{eq: BQP}
\begin{align*}
    0 &=  U(B,Q,P) = rB\left(1 - \frac{Q_m}{Q}\right) h(B) - lB -  \frac{D}{z_m}B \\
    0 &= V(B,Q,P) = \rho(Q,P)-rQ\left(1-\frac{Q_m}{Q}\right)h(B) \\
    0 &= W(B,Q,P) = \frac{D}{z_m}(P_h - P) - \rho(Q,P)B + lBQ.
\end{align*}
This system has been studied extensively in works such as \cite{heggerud2020,wang2007} without the term $lBQ$ on the third equation. The extinction steady state is given by $E_0 = (0,\hat{Q},P_h)$ such that
\begin{align*}
   \hat{Q} &= \frac{\tilde\rho(P_h)Q_M+rQ_mh(0)}{\tilde\rho(P_h)+rh(0)}\\ 
   &= Q_m\frac{rh(0)}{\tilde\rho(P_h)+rh(0)} + Q_M\left(1-\frac{rh(0)}{\tilde\rho(P_h)+rh(0)}\right) \text{  with  } \tilde\rho(P) = \frac{\rho_m}{Q_M-Q_m}\frac{P}{P+M}.
\end{align*} 
Notice that $\hat Q$ is a interpolation between $Q_m $ and $Q_M$. By \cite{wang2007}, define
\begin{align*}
    R_0 = \frac{rh(0)(1-Q_m/\hat{Q})}{l+D/z_m}.
\end{align*}
By following the proof of \cite{wang2007}, one can easily check that if $R_0 <1$, the extinction equilibrium $E_0$ is globally asymptotically stable for the ODE system even with the term $lBQ$. For $R_0 > 1$, $E_0$ is unstable, and it is easily checked that there exists a positive equilibrium $E^*$ of the system with $lBQ$, and cyanobacteria uniformly persist. Note that when $P_h=0$, $\hat Q=Q_m$, the extinction equilibrium $E_0$ is globally asymptotically stable since $R_0=0$. Therefore, the positiveness of $P_h$ is necessary for cyanobacteria to persist uniformly.

\subsubsection{Steady-states with small diffusion and constant advection}

The positive equilibrium \( E^* := (B^*, Q^*, P^*) \) satisfies the system \eqref{eq: BQP}. By multiplying the second equation by \( B^2 \), we find that \( E_0 \) also satisfies a modified version of \eqref{eq: BQP}.  Define the equilibria \( U_0 = (0,0,P_h) \) and \( U^* = (B^*, p^*, Q^*) \), where \( p^* = Q^* B^* \), for the system \eqref{model:main}. However, if the term \( p/B \) appears in the evaluation of \( U_0 \), we use \( \hat{Q} \) instead. Specifically, we assume that if \( U = (B, p, P) \) approaches \( U_0 \) as \( t \to \infty \), then \( p/B \) converges to \( \hat{Q} \) in the same limit.

We linearize system \eqref{model:main} at $\bar U = (\bar B,\bar p,\bar P)$, where $\bar B,\bar p$, and $\bar P$ are nonnegative constants. Let $\tilde{U}=U - \bar U$, and dropping the tilde signs for notational simplicity. Inspired by the transformation described in \cite{Cantrell2004}, let
\begin{equation}
    B(x,t) = \sum_n\hat{B}(t) e^{inx}, \quad p(x,t) = \sum_n\hat{p}(t) e^{inx}, \quad P(x,t) = \sum_n\hat{P}(t) e^{inx},
\end{equation} where $n$ is a mode number that determines the spatial frequency of perturbations.

\begin{figure}[ht]
    \centering
    \begin{subfigure}[b]{\linewidth}
        \centering
        \includegraphics[width=\linewidth]{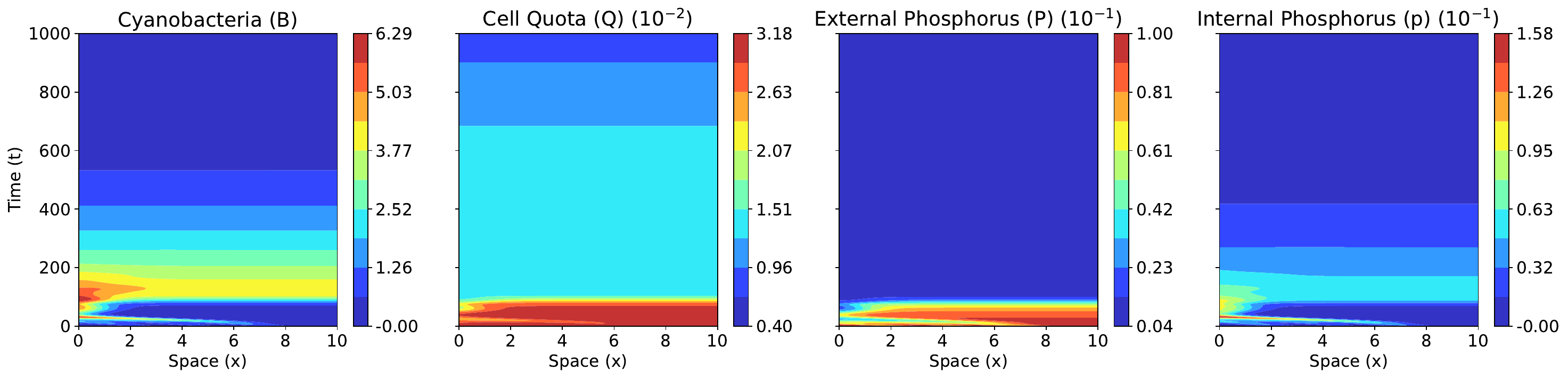}
    \end{subfigure}
    \vskip\baselineskip 
    \begin{subfigure}[b]{\linewidth}
        \centering
        \includegraphics[width=\linewidth]{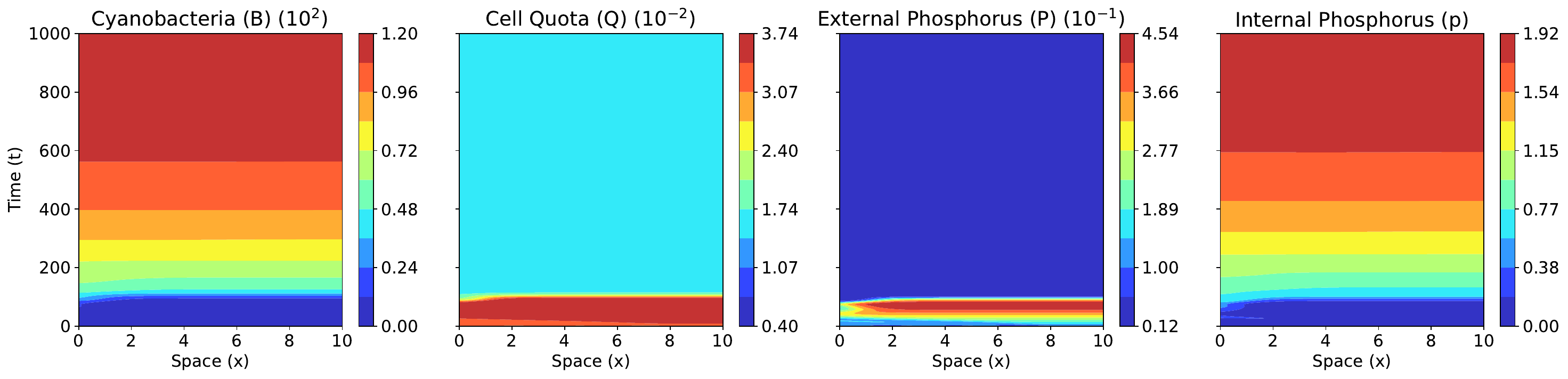}
    \end{subfigure}

    \caption{Time series simulations in one dimension of cyanobacteria, cell quota, and external and internal phosphorus over a 1000-day period. The top row represents simulations when $P_h = 0$, and the bottom when $P_h = 2$. We interpolated a function to a real wind vector from Pigeon Lake, Alberta, in 2023. Once cyanobacteria consume all initial phosphorus, the respective equilibria are stable for all time after.}
    \label{fig:1D_Stability_no_ph}
\end{figure}

\noindent We obtain the transformed system
\begin{equation}
\label{eq:transformed linear system}
    \begin{aligned}
        \hat{B}_t &= -\alpha n^2 \hat{B} - \beta_{B} \mathbf{v} i n \hat{B} + r\left(1 - \frac{2Q_m \bar B}{\bar p }\right) h(\bar B ) \hat{B}  + r\left(1 - \frac{Q_m\bar B}{\bar p}\right) h'(\bar B)\bar B \hat{B}  \\
        & \quad + r\frac{Q_m\bar B^2}{\bar p^2} h(\bar B) \hat{p} - l \hat{B}  -  \frac{D}{z_m} \hat{B}, \\
        \hat{p}_t &= -\alpha n^2 \hat{p}- \beta_B \mathbf{v} i n \hat{p}+ \rho_m \frac{Q_M}{Q_M - Q_m} \frac{\bar P}{\bar P + M} \hat{B} - \frac{\rho_m}{Q_M - Q_m} \frac{\bar P}{\bar P + M}\hat{p} + \rho_m \frac{Q_M \bar B - \bar p}{Q_M - Q_m} \frac{M}{(\bar P + M)^2} \hat{P} \\
        & \quad - l \hat{p} -  \frac{D}{z_m}\hat{p}, \\
        \hat{P}_t &= -\beta n^2 \hat{P} - \beta_P \mathbf{v} i n \hat{P} -  \frac{D}{z_m} \hat{P} - \rho_m \frac{Q_M}{Q_M - Q_m} \frac{\bar P}{\bar P + M} \hat{B} + \frac{\rho_m}{Q_M - Q_m} \frac{\bar P}{\bar P + M}\hat{p} \\
        & \quad - \rho_m \frac{Q_M \bar B - \bar p}{Q_M - Q_m} \frac{M}{(\bar P + M)^2} \hat{P} + l \hat{p}, 
    \end{aligned}
\end{equation} 
Further, the entries of the Jacobian matrix of the transformed equation are given by 
\begin{align*}
    a_{11} &= -\alpha n^2 - \beta_B \mathbf{v} i n - l - \dfrac{D}{z_m} + r\left(1 - \dfrac{2Q_m \bar B}{\bar p}\right) h(\bar B) + r\left(1 - \dfrac{Q_m \bar B}{\bar p}\right) h'(\bar B) \bar B, \\[2ex]
    a_{12} &= r\dfrac{Q_m \bar B^2}{\bar p^2} h(\bar B), \\[2ex]
    a_{21} &= \rho_m \dfrac{Q_M}{Q_M - Q_m} \dfrac{\bar P}{\bar P + M}, \\[2ex]
    a_{22} &= -\alpha n^2 - \beta_B \mathbf{v} i n - l - \dfrac{D}{z_m} - \dfrac{\rho_m}{Q_M - Q_m} \dfrac{\bar P}{\bar P + M}, \\[2ex]
    a_{23} &= \rho_m \dfrac{Q_M \bar B - \bar p}{Q_M - Q_m} \dfrac{M}{(\bar P + M)^2}, \\[2ex]
    a_{31} &= -\rho_m \dfrac{Q_M}{Q_M - Q_m} \dfrac{\bar P}{\bar P + M}=-a_{21}, \\[2ex]
    a_{32} &= \dfrac{\rho_m}{Q_M - Q_m} \dfrac{\bar P}{\bar P + M} + l, \\[2ex]
    a_{33} &= -\beta n^2 - \beta_P \mathbf{v} i n - \dfrac{D}{z_m} - \rho_m \dfrac{Q_M \bar B - \bar p}{Q_M - Q_m} \dfrac{M}{(\bar P + M)^2}.
\end{align*}
\noindent The Jacobian matrix $J$ is also of the form
\[
J = A +\Delta,
\]
where $ A $ is a matrix with $n=0$ and $\Delta=- n^2 \operatorname{diag}(\alpha, \alpha, \beta) - i n \operatorname{diag}(\beta_B, \beta_B, \beta_P)$ with small positive parameters \( \alpha, \beta, \beta_B\), and \( \beta_P \). The imaginary unit is denoted by \( i \), and \( n \) is a natural number. In this regard, we can think of $J$ as a perturbation of $A$ with a small matrix $\Delta$, say $\|\Delta\|\ll\|A\|$. Then the eigenvalues of $J$ can be approximately analyzed using the first-order perturbation theory. Assume that the perturbed eigenvalue $ \mu_i $ and eigenvector $ \mathbf{u}_i $ can be expressed as $ \mu_i = \lambda_i + \epsilon_i $ and $ \mathbf{u}_i = \mathbf{v}_i + \mathbf{w}_i $, where $ \epsilon_i $ and $ \mathbf{w}_i $ represent small first-order corrections. Substituting these into the eigenvalue equation $ J \mathbf{u}_i = \mu_i \mathbf{u}_i $ leads to  
\begin{align*}
(A + \Delta)(\mathbf{v}_i + \mathbf{w}_i) = (\lambda_i + \epsilon_i)(\mathbf{v}_i + \mathbf{w}_i).
\end{align*}
Expanding both sides gives  
\begin{align*}
A\mathbf{v}_i + A\mathbf{w}_i + \Delta\mathbf{v}_i + \Delta\mathbf{w}_i = \lambda_i\mathbf{v}_i + \lambda_i\mathbf{w}_i + \epsilon_i\mathbf{v}_i + \epsilon_i\mathbf{w}_i.
\end{align*}

\begin{figure}[ht]
     \includegraphics[width=0.48\linewidth]{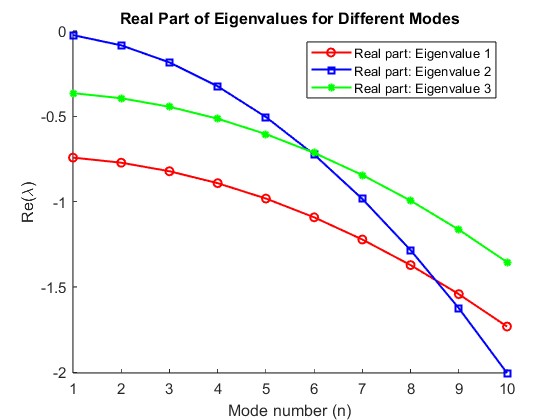}~\includegraphics[width=0.48\linewidth]{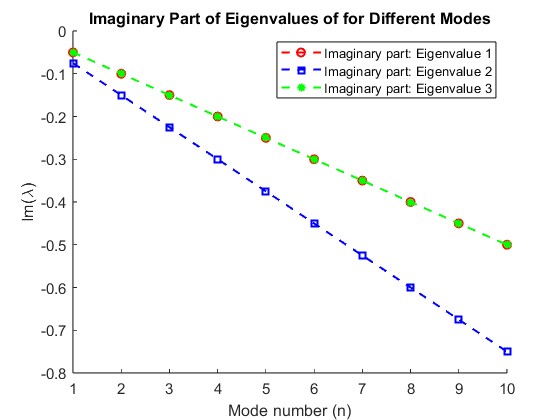}
     \caption{Real and imaginary part of eigenvalues for different modes for the extinction equilibrium $E_0$, respectively. Here we set $r = 0.7$ and $P_h = 0$.}
     \label{fig:extinction}
\end{figure}

\noindent Subtracting $ A\mathbf{v}_i = \lambda_i\mathbf{v}_i $ from both sides results in  
\begin{align*}
A\mathbf{w}_i + \Delta\mathbf{v}_i + \Delta\mathbf{w}_i = \lambda_i\mathbf{w}_i + \epsilon_i\mathbf{v}_i + \epsilon_i\mathbf{w}_i.
\end{align*}
Since $ \Delta\mathbf{w}_i $ and $ \epsilon_i\mathbf{w}_i $ are second-order small terms, they can be ignored, simplifying to  
\begin{align*}
A\mathbf{w}_i + \Delta\mathbf{v}_i \approx \lambda_i\mathbf{w}_i + \epsilon_i\mathbf{v}_i.
\end{align*}
Rearranging yields  
\begin{align*}
(A - \lambda_i I)\mathbf{w}_i \approx -\Delta\mathbf{v}_i + \epsilon_i\mathbf{v}_i.
\end{align*}
Taking the inner product with $ \mathbf{v}_i $ gives  
\begin{align*}
\mathbf{v}_i^* (A - \lambda_i I)\mathbf{w}_i \approx -\mathbf{v}_i^* \Delta\mathbf{v}_i + \epsilon_i \mathbf{v}_i^* \mathbf{v}_i.
\end{align*} 
Since $ (A - \lambda_i I)\mathbf{w}_i $ is orthogonal to $ \mathbf{v}_i $, the left-hand side vanishes. With $ \mathbf{v}_i^* \mathbf{v}_i = 1 $, this simplifies to  
\begin{align*}
\epsilon_i \approx \mathbf{v}_i^* \Delta\mathbf{v}_i.
\end{align*}
Thus, the first-order corrected eigenvalue is  
\begin{align*}
\mu_i \approx \lambda_i + \mathbf{v}_i^* \Delta\mathbf{v}_i.
\end{align*} 
 
\noindent For a diagonal perturbation matrix $ \Delta = \text{diag}(d_1, d_2, d_3) $, this simplifies to  
\begin{align*}
\mathbf{v}_i^* \Delta\mathbf{v}_i = d_1 |v_{i1}|^2 + d_2 |v_{i2}|^2 + d_3 |v_{i3}|^2,
\end{align*}
where $ |v_{ik}|^2 $ represents the squared magnitude of the $ k $-th component of $ \mathbf{v}_i $.  

\begin{figure}[ht]
     \centering
     \includegraphics[width=0.47\linewidth]{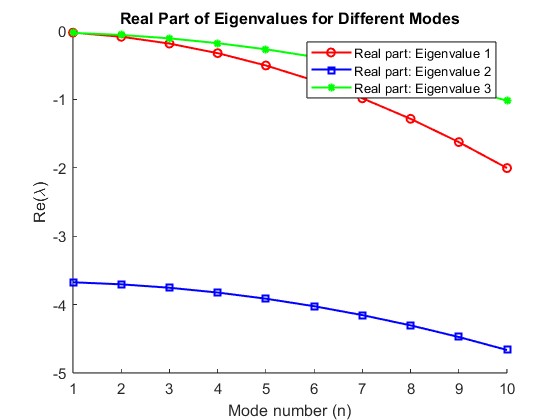}~\includegraphics[width=0.48\linewidth]{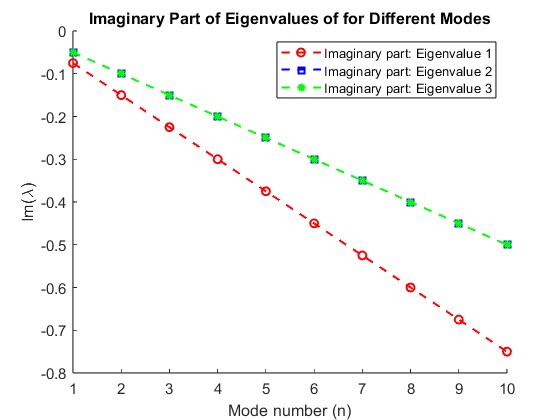}
     \caption{Real and imaginary part of eigenvalues for different modes for the extinction equilibrium $E_0$, respectively. $r=0.7$ and $P_h=0.2$. }
     \label{fig:extinction_positive R0} 
 \end{figure} 

\noindent For the given problem, the perturbation terms are defined as  
\begin{align*}
d_1 = d_2 = -n^2 \alpha - i n \beta_B, \quad d_3 = -n^2 \beta - i n \beta_P.
\end{align*} 

\noindent Thus, the correction becomes  
\begin{align*}
\epsilon_i = -n^2 \left( \alpha (|v_{i1}|^2 + |v_{i2}|^2) + \beta |v_{i3}|^2 \right) - i n \left( \beta_B (|v_{i1}|^2 + |v_{i2}|^2) + \beta_P |v_{i3}|^2 \right).
\end{align*}
 
\noindent Therefore, the first-order corrected eigenvalues of $ J $ are  
\begin{align*}
\mu_i \approx \lambda_i -n^2 \left( \alpha (|v_{i1}|^2 + |v_{i2}|^2) + \beta |v_{i3}|^2 \right) - i n \left( \beta_B (|v_{i1}|^2 + |v_{i2}|^2) + \beta_P |v_{i3}|^2 \right)
\end{align*} 
where $ |v_{ij}|^2 $ represents the squared magnitude of the $j$ th component of the eigenvector $ \mathbf{v}_i $ associated with $ A $.  
As a result, eigenvalues of $J$ can be approximated as
\begin{equation*}
    \text{Re}(\mu_i) \approx \text{Re}(\lambda_i) - n^2(\alpha v_{i1}^2+\alpha v_{i2}^2 +\beta v_{i3}^2) \quad \text{ and }\quad \text{Im}(\mu_i) \approx \text{Im}(\lambda_i)-n(\beta_B v_{i1}^2 + \beta_B v_{i2}^2 + \beta_P v_{i3}^2).
\end{equation*}

The stability of the ODE system is determined by $\lambda_i$. Therefore, we observe that if constant equilibrium $E$ of ODE system is stable and the parameters $\alpha,\beta,\beta_B,\beta_P$ are small, then corresponding homogeneous equilibrium $U$ of PDE system is also stable. Otherwise, there exists $n$ for $U$ to be unstable.

\vspace{2pt}

\noindent We find the exact value of $U^*$ numerically from ODE system for \eqref{model:main} with $\alpha=\beta=0$. For the eigenvalue simulations performed in MATLAB, we take parameters to be
$\alpha = 0.01, \beta = 0.02, \beta_B = 0.05, \beta_P = 0.075, z_m = 5, Q_m = 0.004, Q_M = 0.04, K_{bg} = 0.3, k = 0.0004, I_{in} = 300, H = 120, l = 0.35, D = 0.02, \rho_m = 1, M = 1.5$. We take different $r$ and $P_h$ to consider the following cases:
\begin{itemize}

\item[(1)]  Stability of $U_0=(0,0,\hat{Q})$ when $r=0.7$ and $P_h=0$, so $R_0=0$. When $P_h=0$, we have $R_0=0$, so that the extinction equilibrium is  expected to be stable. We observe negative real parts of eigenvalues for all nodes. See Figure \ref{fig:extinction}. 

\item [(2)] Stability of $U_0=(0,0,\hat{Q})$ when $r=0.7$ and $P_h=0.2$, so  $R_0=0.9494$. See Figure~\ref{fig:extinction_positive R0}.

\item [(3)] Stability of $U^*=(B^*,p^*,P^*)$ when $r=1$, $P_h=0.2$ and $R_0=1.3497>1$. In this case, $U_0$ is linearly unstable, and there is a positive linearly stable  equilibrium $U^*=(16.2785,0.1920    ,0.0080)$, which is found by a numerical simulation of the ODE system for \eqref{model:main}. See Figure~\ref{fig:E0 R0_1.3497} and \ref{fig:Positive equilibrium}.

\begin{figure}[ht]
     \centering
     \includegraphics[width=0.48\linewidth]{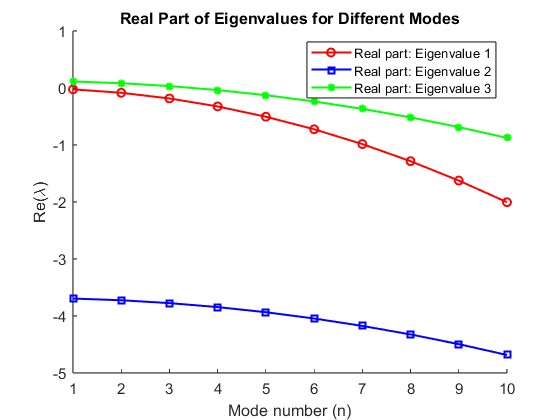}~
     \includegraphics[width=0.48\linewidth]{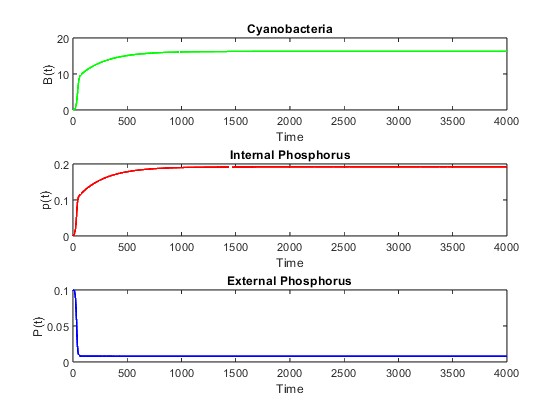}
     \caption{The right figure represents the real part of eigenvalues for different modes for the positive equilibrium $U_0$ when $r=1$ and $P_h=0.2$. The third eigenvalue of $J$ around $U_0$ is positive for $n=1,2$, and 3. With the same parameter, the right figure illustrates the solution of ODE system for \eqref{model:main}. The solution at time $4000$ is considered to be the equilibrium $U^*:=(16.2785,0.1920    ,0.0080)$.}
     \label{fig:E0 R0_1.3497}
\end{figure}

\begin{figure}[ht]
     \centering
     \includegraphics[width=0.48\linewidth]{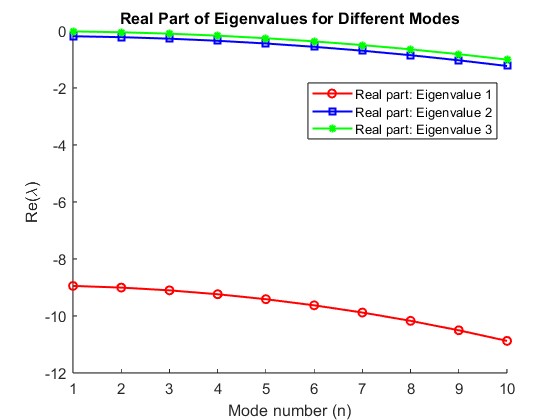}~
     \includegraphics[width=0.48\linewidth]{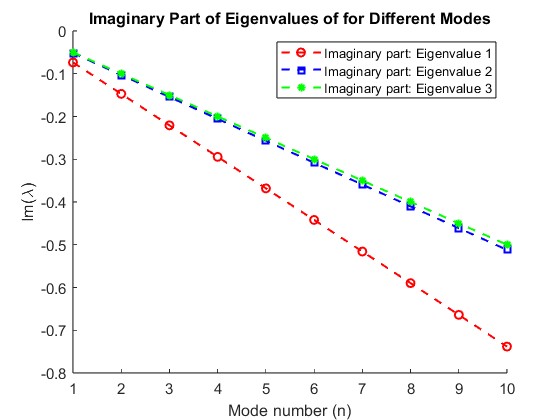}
     \caption{Real and imaginary part of eigenvalues for different modes for the positive equilibrium $U^*=(16.2785, 0.1920, 0.0080)$, respectively. $r=1$ and $P_h=0.2$.}
     \label{fig:Positive equilibrium}
\end{figure}
\end{itemize}

We have observed the stability of the extinction equilibrium $U_0$ and the positive equilibrium $U^*$ numerically. As we found in the first-order perturbation theory, the real part of the numerical eigenvalues are quadratic polynomials in terms of $n$. This implies that a small diffusion of individuals ends up stabilizing our system. The small random movement of cyanobacteria in lakes and minor horizontal movement caused by wind or currents contribute to stabilizing the overall system. This leads us to the following question: What factors drive the spatial patterns of cyanobacteria observed in lakes? Since the random movement of cyanobacteria is generally small, we can consider several possible factors: (1) Movement induced by lake water currents - Unlike the mathematical analysis we performed, the strength and direction of the wind varies spatially and changes over time. As a result, heterogeneous changes in currents may influence the spatial patterns for large time. (2) Variations in vertical exchange across space - Even within a lake, convective processes and water depth differ across locations, leading to variations in the amount of vertical exchange.

\begin{figure}[ht]
    \centering
    \includegraphics[width=\linewidth]{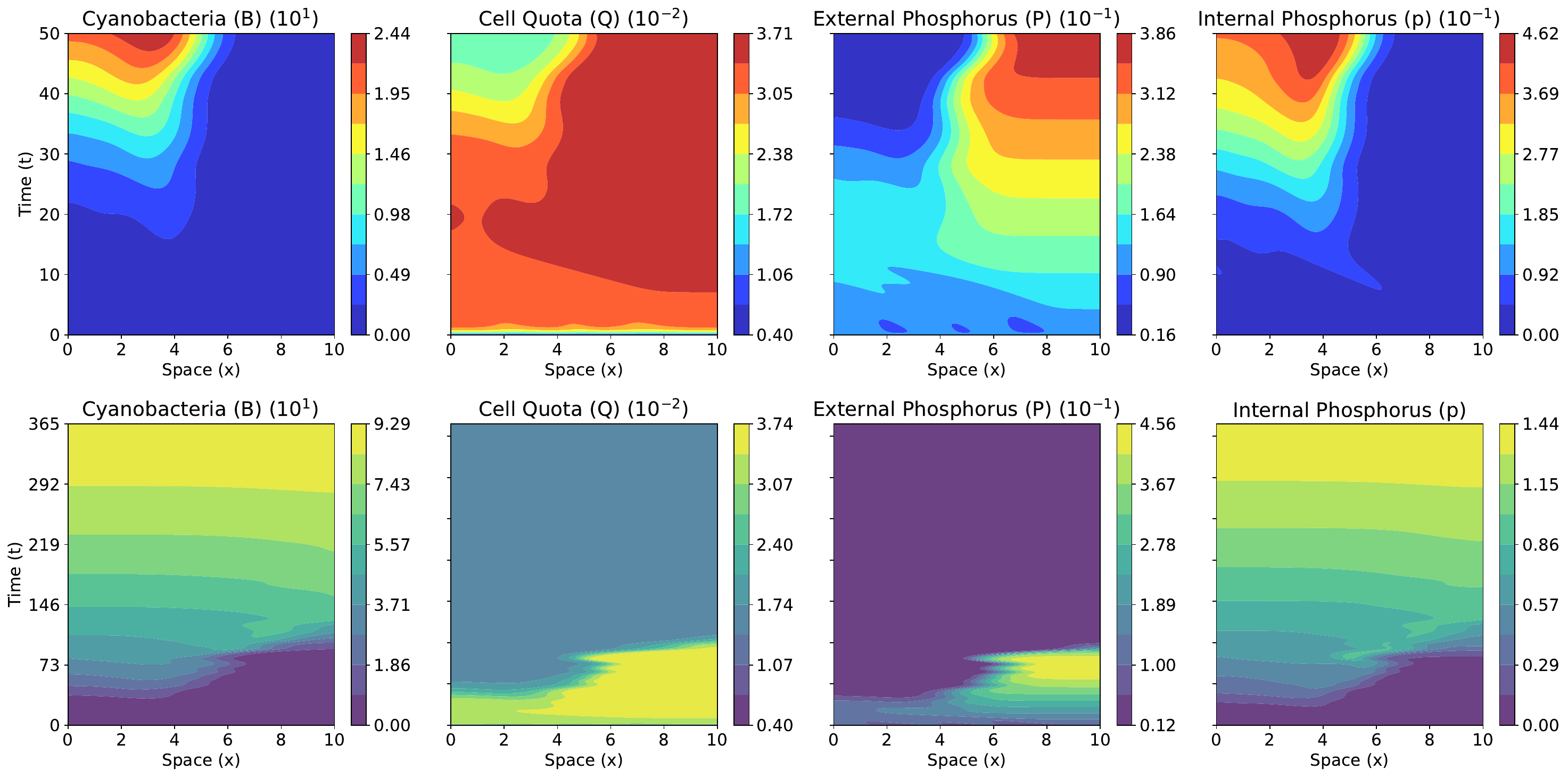}
    \caption{Time series simulations in one dimension of cyanobacteria, cell quota, external and internal phosphorus over short (50 days) and long (365 days) time periods. We interpolated a function to a real wind vector from Pigeon Lake, Alberta, in 2023.}
    \label{fig:1D_Time_Series_Real_Wind}
\end{figure}

\subsection{Time Series Simulations}
In this section, we conduct numerical simulations of system (\ref{model:main}) with the addition of $Q$. We will observe the dynamics of our state variables in both short and long-term timescales. We consider discrete diffusion and advection in space. All simulations are performed in Python. Details can be found in Appendix \ref{ap: numerics}.

We consider a one-dimensional spatial domain $x \in [0,L]$ of model (\ref{model:main}) with the addition of $Q$. We impose Neumann boundary conditions at $x=0$ and $x=L$. The parameters and functions are as described in the model formulation. Using the finite difference method, we discretize the spatial domain and for any spatially dependent variable $U(x,t)$, let $U_i(t) \approx U(x_i,t)$. The first- and second-order spatial derivatives are approximated by central difference schemes, except advection, which is approximated by an upwind scheme. This described discretization transforms the system into a system of ODEs in time for the arrays $B_i(t)$, $Q_i(t)$, $P_i(t)$, and $p_i(t)$. We then integrate the resulting ODE system in time using the BDF method built into Python. Specifically, we observe short (transient) and medium timescales in Figures \ref{fig:1D_Time_Series_Real_Wind}, \ref{fig:1D_Time_Series_No_Wind}, \ref{fig:1D_Time_Series_Fake_Wind}. To observe the asymptotic dynamics of our model we also ran the simulations for 1000 days in Figure $\ref{fig:1D_Stability_no_ph}$.

In addition to the one-dimensional simulations, we numerically solved \eqref{model:main} on a realistic lake-shaped domain using the Firedrake package to perform finite element method (FEM) simulations (\cite{FiredrakeUserManual}). Specifically, we rendered a closed curve of the boundary of Pigeon Lake, Alberta, Canada, from satellite imagery provided by the government of Alberta (\cite{PigeonLake}) and then generated an unstructured triangular mesh of approximately 5,000 elements spanning this domain. The parameters are identical to the one-dimensional simulations, and advection is given by a two-dimensional time-dependent wind vector from local meteorological data in 2023, which we then interpolated. We updated $\vec{v}(t)$ accordingly for each time step to capture wind-driven surface mixing. The finite elements implementation follows our weak formulation from Appendix~\ref{ap:weak_form}, discretizing the domain with Galerkin polynomials for $B,\,p,\,P,$ and $Q$. As shown in Figure~\ref{fig:2D_time_series}, the distribution of cyanobacteria evolves in spatially heterogeneous ways. After about 50-75 days, solutions concentrate in wind-protected inlets if $P_h>0$; otherwise, we observe near-complete depletion away from inflow points when $P_h=0$. These results adhere to our theoretical threshold found in Section \ref{sec:analysis}. 

\subsection{Sobol Indices}

We performed a global sensitivity analysis to investigate how our model's physical lake condition parameters influence cyanobacterial biomass predictions over the simulated time window. In particular, we assessed the relative importance of the following parameters: epilimnion depth $(z_m)$, background light attenuation $(K_{bg})$, water exchange rate $(D)$, external phosphorus input $(p_{in})$, and the cyanobacterial advection coefficient $(\beta_B)$. We utilized Sobol's method (\cite{sobol2001global}), which uses Monte Carlo sampling to evaluate the contributions of individual factors and their interactions in producing the overall variability in the cyanobacteria solution. Sobol's method provides two indices: the first-order Sobol index $S1_f$, which measures how much of the output variance is attributable to a factor $f$ alone, and the total-order Sobol index $ST_f$, which captures both the direct effect of $f$ and all its interactions with the other parameters. Further details can be found in Appendix \ref{ap: numerics}. 

\section{Discussion}

Cyanobacterial blooms (CBs) resulting from complex interactions between nutrient availability and physical lake characteristics are responsible for numerous economic and environmental problems, such as biodiversity loss and water contamination (\cite{paerl2013harmful}). In this paper, cyanobacteria take nutrients from the surrounding water and modulate their internal phosphorus quota under varying light conditions, thereby influencing their growth and spatial distribution. To better understand these dynamics, we developed a reaction-diffusion–advection model that integrates ecological stoichiometry with realistic wind conditions. This methodology enabled us to rigorously characterize how heterogeneity in nutrient inputs and water movement impacts bloom formation and persistence. Theoretical outcomes show that model (\ref{model:main}) behaves well both mathematically and biologically. According to the results in Section \ref{sec:analysis} and corresponding remarks, some critical thresholds for CBs are rigorously derived. Specifically, we observe cyanobacteria's basic ecological reproductive indices (see Section \ref{sec:analysis}) through various transformations and classical theoretical techniques.

After approximately 300 days, each system—no movement, synthetic wind, and real data-driven wind—displays qualitatively similar behaviour for each set of one-dimensional simulations. This phenomenon occurs despite heterogeneous initial conditions because cyanobacteria eventually saturate the available space during a bloom. Notably, while differences in the magnitude of external phosphorus are evident across the movement-based simulations, the extreme values of cyanobacteria concentration, internal phosphorus, and thus cell quota remain relatively consistent. In other words, once the cyanobacterial population exceeds a critical threshold, advective and diffusive movement exerts a diminished influence on overall dynamics (\cite{huisman1994, paerl2008blooms}). Our two-dimensional simulations on a lake-shaped domain show the complex interactions of local land geometry, wind-driven advection, and nutrient stoichiometry related to CBs. The irregular boundary geometry produces concentration hot spots along sharply curving shorelines (Figure \ref{fig:2D_time_series}). This phenomenon partially corroborates earlier one-dimensional findings- once a critical threshold of dissolved phosphorus is met, strong gradients in cyanobacteria emerge, particularly where wind currents stagnate or recirculate. From a water-resource management viewpoint, these spatially explicit results emphasize the utility of controlling inflow nutrients in shallow or sheltered segments of the lake boundary, where wind-driven advection may be insufficient to dilute bloom-prone regions.

To further investigate the steady-state behaviour of our model, we simulated the system for 1000 days using a finite difference scheme (Figure \ref{fig:1D_Stability_no_ph}). When \(P_h = C > 0\), after the cyanobacteria deplete the initial phosphorus supply, the model settles into an internally maintained constant steady state that is asymptotically stable for all future time (Figure \ref{fig:1D_Stability_no_ph}). Conversely, if \(P_h = 0\), once the phosphorus is exhausted, the extinction equilibrium is reached and remains stable. In dynamical systems terms, \(P_h = 0\) represents a bifurcation point, marking a qualitative change in the system’s long-term behaviour. These outcomes are consistent across all simulation scenarios and align with the analysis presented in Section~3.4. These results highlight the need to monitor nutrient inputs from anthropogenic sources closely (\cite{zhang2022, zhao2014, paerl2013harmful}).

Our global sensitivity analysis, implemented via Sobol indices (\cite{sobol2001global, herman2017salib}), reveals that light attenuation and intensity consistently dominate cyanobacterial growth. In contrast, epilimnion depth and water exchange rate exert a relatively constant yet minor influence over the simulation period. Interestingly, the relative impact of input phosphorus becomes more pronounced over time. This trend resonates with previous observational and modelling studies (\cite{gao2020, du2022, reynolds2006}). Moreover, the advection coefficient proves significantly more influential when the model incorporates real wind data than an idealized oscillatory wind function. This is likely due to the spatial symmetry induced by the sinusoidal forcing, which tends to homogenize the initial conditions before bloom formation (\cite{cao2006, huisman2018cyanobacterial}). Overall, these findings reinforce the dominant role of light and nutrient dynamics in bloom development and highlight the importance of incorporating realistic physical forcing in future modelling efforts (\cite{qin2009, paerl2008blooms}).

In model (\ref{model:main}), we do not consider the presence of any predatory species. Introducing a grazer into the system could provide insights into how spatial heterogeneity affects toxin dispersal and bioaccumulation in other species at higher trophic levels. Further, it may be helpful to incorporate a third spatial dimension in the form of water depth to consider other species, such as algae or fish. The explicit calculation of the internal steady state could benefit scientists and policymakers by providing precise guidelines for the quantity of cyanobacteria, given the different parameter values. Most existing reaction-diffusion models incorporating nutrient stoichiometry via a Droop formulation focus on diffusion in one-dimensional space, with limited consideration of reaction terms beyond growth (\cite{hsu2010,hsu2014,Hsu2017}). By contrast, our work integrates wind-driven advection and nutrient stoichiometry in a two-dimensional setting to capture more realistic bloom dynamics. This approach better explains how environmental forcing and boundary geometry interact to shape CBs, providing insights that can inform future research and resource management strategies.

\begin{figure}[ht]
    \centering
    \includegraphics[width=\linewidth]{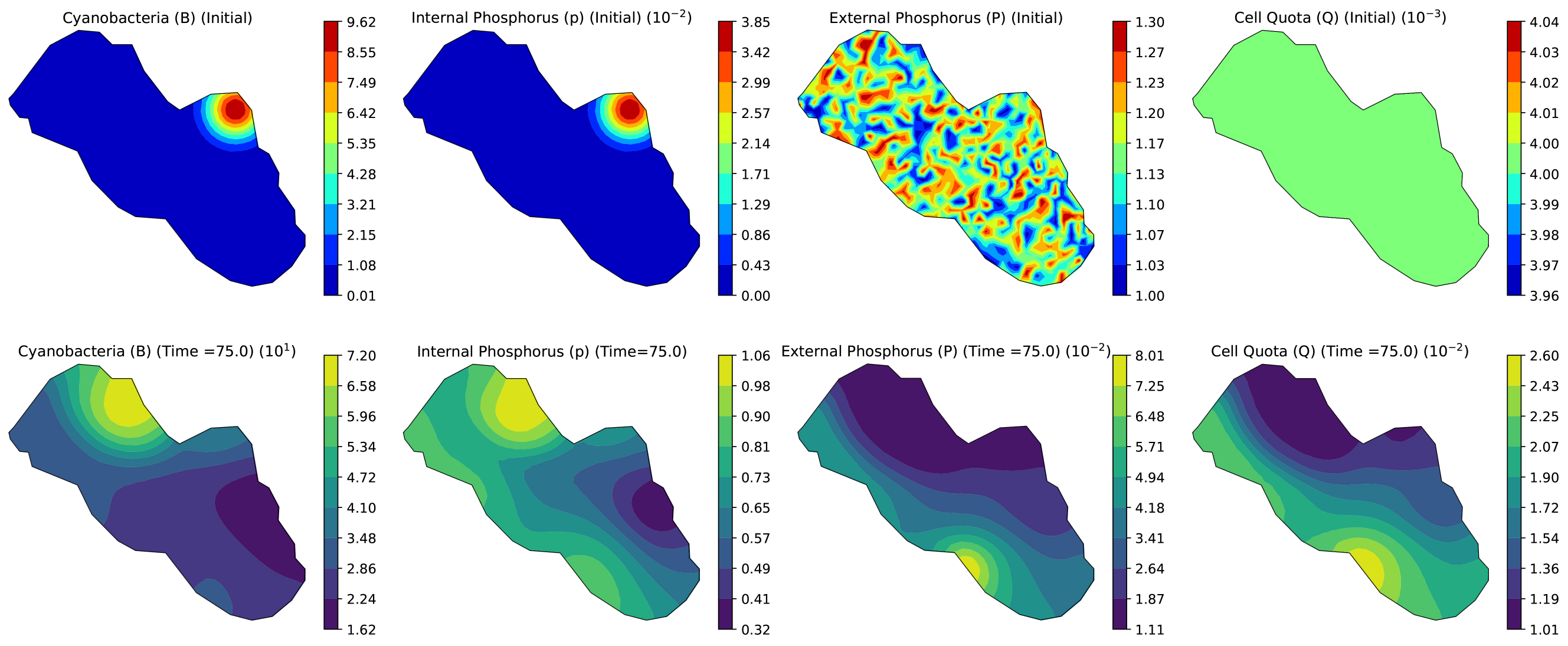}
    \caption{Time series simulations in two dimensions of cyanobacteria, cell quota, external and internal phosphorus over some day time period when $P_h = 2$. We interpolated a function to a real wind vector from Pigeon Lake, Alberta, in 2023. We used the finite elements method and satellite image data to simulate the system on a realistic lake domain.}
    \label{fig:2D_time_series}
\end{figure}

\begin{figure}[ht]
    \centering
    \includegraphics[width=\linewidth]{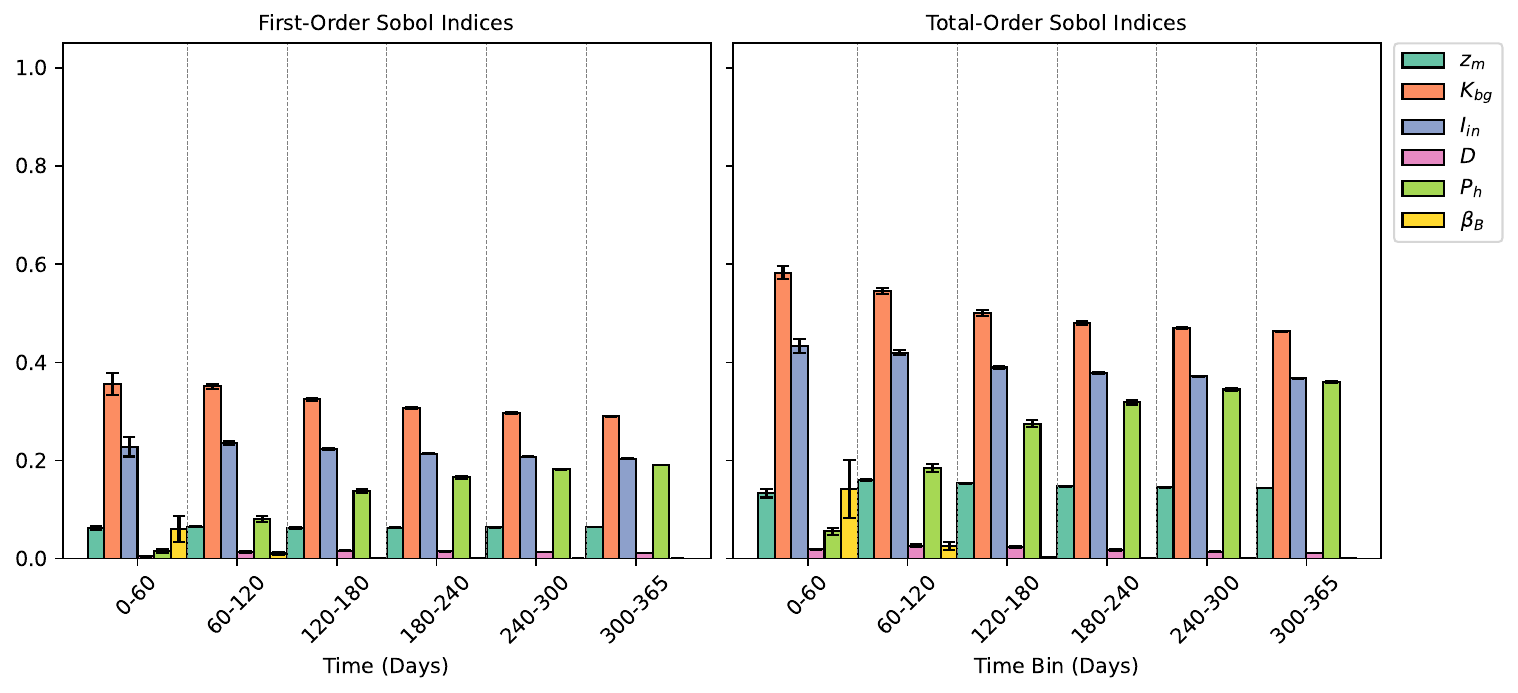}
    \caption{Global sensitivity analysis using first- and total-order Sobol indices to compare the impact of physical lake conditions on the growth of cyanobacteria over space and time. The simulation is for 365 days, with contributions binned into ~60 day intervals. We interpolated a function to a real wind vector from Pigeon Lake, Alberta, in 2023.}
    \label{fig: sobol_real_wind}
\end{figure}

\clearpage
\newpage

\appendix
\section{\texorpdfstring{Formal derivation of dispersal of $B,p,Q$}{Formal derivation of dispersal of B, p, Q}}

\label{ap: deriveQdiffusion}
Let us consider a discrete time and space model on one-dimensional space.
Let $p_{n}^{i}$ and $B_{n}^{i}$ be the number of internal phosphorus and cyanobacteria at location $i$ and time $n$, respectively. Let us denote by the cell quota $Q_{n}^{i} := \frac{p_{n}^{i}}{B_{n}^{i}}$. We assume that if cyanobacteria moves randomly with the dispersal rate $d$ then the internal phosphorus $p$ also moves along the same path for the bacteria with the same rate $d$. 
For each $i$-th patch, we can calculate
\begin{equation*}
\begin{aligned}
    &B_{n+1}^{i} = B^{i}_n +d_{ii+1}B_{n}^{i+1} - d_{i+1i}B_{n}^{i} - d_{i-1i}B_{n}^{i} + d_{i-1i}B_{n}^{i-1}\\
    &p_{n+1}^{i} = p^{i}_n + d_{ii+1}p_{n}^{i+1} - d_{i+1i}p_{n}^{i} - d_{i-1i}p_{n}^{i} + d_{i-1i}p_{n}^{i-1},
\end{aligned} 
\end{equation*} where $d_{ji}$ is a dispersal rate at which bacteria move from the $i$-th patch to the $j$-th patch. Cyanobacteria are mainly dispersed through ocean currents, wind-driven mixing, and wave action. These passive processes allow their population to spread over vast areas, reducing local competition for vital resources like nutrients and sunlight. By distributing across diverse environments, cyanobacteria minimize the risk of localized resource depletion or environmental stresses that could lead to population collapse.
In this context, we assume that the dispersal rate $d_{ji}$ of cyanobacteria follows Fick's law, which assumes uniform dispersal between patches. Denoting by $d_{i+1i}=d_{ii+1}=d_{i+1/2}$ for each $i$, the resulting equations approximate diffusion processes of the Fickian type,
\begin{align*}
    B_t = (d(x)B_x)_{x}\quad \text{ and }\quad
    p_t = (d(x)p_x)_{x}.
\end{align*}

In order to explain a diffusive movement of the ratio $Q$ when $B$ and $p$ diffuse, it is necessary to derive a diffusion equation of the cell quota $Q_n^i$. Indeed,
 since $p_{n+1}^i=Q_{n+1}^i B_{n+1}^i$, 
\begin{align*}
    Q_{n+1}^iB_{n+1}^i &= p_{n+1}^i \notag\\
    &= p_{n}^i+d_{i+1/2}(p_{n}^{i+1} - p_{n}^{i})  - d_{i-1/2}(p_{n}^{i} - p_{n}^{i-1})\notag\\
    &= Q_{n}^iB_{n}^i+d_{i+1/2}(Q_{n}^{i+1}B_{n}^{i+1} - Q_{n}^{i}B_{n}^{i})  - d_{i-1/2}(Q_{n}^{i}B_{n}^{i} - Q_{n}^{i-1}B_{n}^{i-1})\notag
\end{align*} Subtracting $Q_n^iB^i_{n+1}$ both sides gives
\begin{align*}
    (Q_{n+1}^{i}-Q_{n}^{i})B_{n+1}^i&= -Q_{n}^i(d_{i+1/2}(B_{n}^{i+1} - B_{n}^{i})  - d_{i-1/2}(B_{n}^{i} - B_{n}^{i-1}))\\&+d_{i+1/2}(Q_{n}^{i+1}B_{n}^{i+1} - Q_{n}^{i}B_{n}^{i})  - d_{i-1/2}(Q_{n}^{i}B_{n}^{i} - Q_{n}^{i-1}B_{n}^{i-1})\notag
\end{align*} where we have used that $B_{n+1}^i - B^{i}_{n}= d_{i+1/2}(B_{n}^{i+1} - B_{n}^{i})  - d_{i-1/2}(B_{n}^{i} - B_{n}^{i-1})$.
Then we can approximate the equation to a continuous equation
\begin{equation*}
\begin{aligned}
    B\partial_t Q &= -Q(d(x)B_x)_{x}+(d(x)(BQ)_x)_{x},\\
    &= B(d(x)Q_x)_x + 2d(x)B_xQ_x\end{aligned}
\end{equation*} so that
\begin{equation}\label{eqn:Qdiffusion}
     \partial_t Q = (d(x)Q_x)_x +2\frac{d(x)}{B}B_xQ_x = \frac{1}{B^2}(d(x)B^2Q_x)_x
\end{equation} where $d(x)$ is the dispersal rate of cyanobacteria. It is given by a positive function in terms of spatial variable $x$. 
We also derive \eqref{eqn:Qdiffusion} with a continuous model. We assume to have $\partial_tB=\nabla\cdot (d(x)\nabla B)$ and $\partial_tp=\nabla\cdot(d(x) \nabla p)$. We introduce the variable $Q=\frac{p}{B}$ and derive
\begin{align*}
    \frac{\partial Q}{\partial t} = \frac{\partial (p/B)}{\partial t} &= \left(\frac{\nabla\cdot(d(x) \nabla p)}{B}-\frac{p}{B^2}\nabla\cdot(d(x) \nabla B))\right) \\
    &= B^{-1}\big(\nabla\cdot(d(x) \nabla (BQ)) -  Q\nabla\cdot(d(x) \nabla B) \big)\\
    &=B^{-1}\big(\nabla\cdot(d(x)B\nabla Q) + d(x)\nabla B\cdot\nabla Q\big)\\
    &=\nabla\cdot(d(x)\nabla Q) + \frac{2}{B}\nabla (d(x)B)\cdot\nabla Q\\
    &= \frac{1}{B^2}\nabla\cdot(d(x)B^2\nabla Q).\\
\end{align*} 

\noindent If $d(x)\equiv \alpha$ and $B$ is constant, we have 
\[  \partial_t Q = \alpha Q_{xx}. \]
Notice that the heterogeneity in the density of cyanobacteria gives rise to failure of conservation law, which is represented by the denominator $B^2$ in \eqref{eqn:Qdiffusion}.

\section{Numerical supplementation}\label{ap: numerics}

\subsection{One dimensional finite difference}

The one-dimensional version of model (\ref{model:main}) is given by
\begin{equation}
\label{model:1D}
\begin{aligned}
\frac{\partial B}{\partial t} &= \alpha \frac{\partial^2 B}{\partial x^2} - \beta_B v(t)\frac{\partial B}{\partial x} \;+\; r B\left(1 - \frac{Q_m}{Q}\right)h(B) \;-\; l B \;-\; \frac{D}{z_m}B, \\[6pt]
\frac{\partial Q}{\partial t} &= \alpha \frac{\partial^2 Q}{\partial x^2} \;+\; \biggl(2\alpha\frac{\partial B/\partial x}{B} - \beta_B v(t)\biggr)\frac{\partial Q}{\partial x} \;+\; \rho(Q,P) \;-\; r Q\left(1-\frac{Q_m}{Q}\right)h(B), \\[6pt]
\frac{\partial P}{\partial t} &= \beta \frac{\partial^2 P}{\partial x^2} - \beta_P v(t)\frac{\partial P}{\partial x} \;+\; \frac{D}{z_m}(P_h - P) \;-\; \rho(Q,P)B \;+\; lQB, \\[6pt]
\frac{\partial p}{\partial t} &= \alpha \frac{\partial^2 p}{\partial x^2} - \beta_B v(t)\frac{\partial p}{\partial x} \;+\; \eta(B,p,P) \;-\; l p \;-\; \frac{D}{z_m}p.
\end{aligned}
\end{equation}

\noindent We impose Neumann boundary conditions at $x=0$ and $x=L$:
\begin{equation}
\frac{\partial B}{\partial x} = \frac{\partial Q}{\partial x} = \frac{\partial P}{\partial x} = \frac{\partial p}{\partial x} = 0 \quad \text{at } x=0, L.
\end{equation}

\noindent Using the finite difference method, we discretize the spatial domain into $N_x$ uniform grid points, where
\[
x_i = i\Delta x,\quad i=0,1,\dots,N_x-1,\quad \text{where}\ \Delta x = \frac{L}{N_x-1}.
\]
For any spatially dependent variable $U(x,t)$, let $U_i(t) \approx U(x_i,t)$. The second-order spatial derivatives are approximated by second-order central differences
\[
\frac{\partial^2 U}{\partial x^2}(x_i,t) \approx \frac{U_{i+1}(t) - 2U_i(t) + U_{i-1}(t)}{\Delta x^2}.
\]
First-order spatial derivatives are approximated by a second-order central difference
\[
\frac{\partial U}{\partial x}(x_i,t) \approx \frac{U_{i+1}(t) - U_{i-1}(t)}{2\Delta x}.
\]

\noindent At the boundaries ($i=0$ and $i=N_x-1$), we apply the Neumann conditions
\[
\frac{\partial U}{\partial x}(x_0,t) \approx \frac{U_1(t)-U_0(t)}{\Delta x} = 0 \implies U_1(t)=U_0(t),
\]
and similarly at $x_{N_x-1}$
\[
\frac{\partial U}{\partial x}(x_{N_x-1},t) \approx \frac{U_{N_x-1}(t)-U_{N_x-2}(t)}{\Delta x} = 0 \implies U_{N_x-1}(t)=U_{N_x-2}(t).
\]

\noindent This discretization transforms the PDE system into a system of ODEs in time for the arrays $B_i(t)$, $Q_i(t)$, $P_i(t)$, and $p_i(t)$. We then integrate the resulting ODE system in time using the BDF method built into Python.

\subsection{Two-dimensional finite elements}
\label{ap:weak_form}

Here, we detail the derivation of the weak formulation for our stoichiometric PDE system \eqref{model:main} on a two-dimensional lake-shaped domain~\(\Omega \subset \mathbb{R}^2\). Recall that \(\partial \Omega\) denotes the boundary of the lake, where we impose homogeneous Neumann conditions. Let \(T>0\) be the final time of interest. Further, suppose 
\begin{align*}
B(\cdot,t),\;p(\cdot,t),\;P(\cdot,t)\;\in\;C\bigl([\,0,T\,],\,H^1(\Omega)\bigr)\;\cap\;L^2\bigl(0,T;\,H^2(\Omega)\bigr),
\end{align*}
with
\begin{align*}
\frac{\partial B}{\partial n} \;=\;\frac{\partial p}{\partial n} \;=\;\frac{\partial P}{\partial n} \;=\;0
\quad\text{on}\quad \partial\Omega\times(0,T),
\end{align*}
where \(n\) is the outward normal. We define
\[
V \;=\;\Bigl\{\, \phi \in H^1(\Omega)\;:\;\partial_n \phi \big|_{\partial\Omega}=0\Bigr\}.
\]
We will write \(v_1,v_2,v_3\) for test functions associated with \(B,p,P\), respectively. Given
\[
\frac{\partial B}{\partial t} -\alpha\,\Delta B +\beta_B\,\vec{v}(t)\cdot\nabla B 
\;=\; r\Bigl(1-\frac{Q_m\,B}{p+\varepsilon}\Bigr)h(B)\,B - l\,B -\frac{D}{z_m}\,B.
\]
Multiply both sides by a test function \(v_1\in V\) and integrate over \(\Omega\):
\begin{align}
&\int_{\Omega}\,\frac{\partial B}{\partial t}\,v_1\,d\mathbf{x}
\;-\;\alpha\int_{\Omega}\,\Delta B\,v_1\,d\mathbf{x}
\;+\;\beta_B\int_{\Omega}\,\bigl(\vec{v}\cdot\nabla B\bigr)\,v_1\,d\mathbf{x}
\;=\;\int_{\Omega}\Bigl[r\Bigl(1-\frac{Q_m\,B}{p+\varepsilon}\Bigr)h(B)\,B - lB - \tfrac{D}{z_m}B\Bigr]v_1\,d\mathbf{x} \\
\implies &\int_{\Omega}\,\frac{\partial B}{\partial t}\,v_1\,d\mathbf{x}
\;+\;\alpha\int_{\Omega}\,\nabla B\cdot\nabla v_1\,d\mathbf{x}
\;+\;\beta_B\int_{\Omega}\,\bigl(\vec{v}\cdot\nabla B\bigr)\,v_1\,d\mathbf{x}
\;=\;\int_{\Omega}\Bigl[r\Bigl(1-\tfrac{Q_m\,B}{p+\varepsilon}\Bigr)h(B)\,B - l\,B -\tfrac{D}{z_m}B\Bigr]\,v_1\,d\mathbf{x}.
\end{align}

We perform a similar method for \(p\) and \(P\). Taking \(v_2\in V\) as the test function for the second equation, we see that
\[
\int_{\Omega}\,\frac{\partial p}{\partial t}\,v_2\,d\mathbf{x}
\;+\;\alpha\int_{\Omega}\,\nabla p\cdot\nabla v_2\,d\mathbf{x}
\;+\;\beta_B\int_{\Omega}\,\bigl(\vec{v}\cdot\nabla p\bigr)\,v_2\,d\mathbf{x}
\;=\;\int_{\Omega}\,\Bigl[\eta(B,p,P)\;-\;l\,p\;-\;\tfrac{D}{z_m}p\Bigr]\,v_2\,d\mathbf{x}.
\]
and
\begin{align*}
\int_{\Omega}\,\frac{\partial P}{\partial t}\,v_3\,d\mathbf{x}
&\;+\;\beta\int_{\Omega}\,\nabla P\cdot\nabla v_3\,d\mathbf{x}
\;+\;\beta_P\int_{\Omega}\,\bigl(\vec{v}\cdot\nabla P\bigr)\,v_3\,d\mathbf{x}
=\;\int_{\Omega}\,\Bigl[\tfrac{D}{z_m}\bigl(P_h - P\bigr)\;-\;\eta(B,p,P)\,B \;+\;l\,p\Bigr]\,v_3\,d\mathbf{x}.
\end{align*}

Let \(\Delta t>0\) be a uniform time step, and define times \(t^n = n\,\Delta t\) for \(n=0,1,\dots,N\) such that \(N\Delta t=T\). A fully discrete backward Euler method for \(B\) is:
\[
\int_{\Omega}\,\frac{\,B^{n+1}-B^n\,}{\,\Delta t\,}\,v_1\,d\mathbf{x}
\;+\;\alpha\int_{\Omega}\,\nabla B^{n+1}\cdot\nabla v_1\,d\mathbf{x}
\;+\;\beta_B\int_{\Omega}\,\bigl(\vec{v}^{n+1}\cdot\nabla B^{n+1}\bigr)\,v_1\,d\mathbf{x}
\;=\;\int_{\Omega}R_B\bigl(B^{n+1},p^{n+1}\bigr)\,v_1\,d\mathbf{x},
\]
where \(R_B(\cdot)\) indicates the reaction terms 
\[
R_B(B^{n+1},p^{n+1}) \;=\;r\bigl(1-\tfrac{Q_m\,B^{n+1}}{\,p^{n+1}+\varepsilon}\bigr)\,h(B^{n+1})\,B^{n+1} \;-\;l\,B^{n+1} \;-\;\tfrac{D}{z_m}\,B^{n+1}.
\]
The forms for \(p\) and \(P\) are identical in structure, each having the discrete-time derivative, diffusion, advection, and reaction integrated against test functions. This set of discrete equations is then solved at each time step \(n\) using a nonlinear solver. The updated solutions \(\bigl(B^{n+1},p^{n+1},P^{n+1}\bigr)\) become the initial condition for the next step. After the final step \(N\), we evaluate \(Q^{n+1} = p^{n+1}/B^{n+1}\) in a pointwise manner. We took the shoreline of the Pigeon Lake, Alberta, Canada provided by the provincial government in EPSG: 3400 system, and then we re-projected into  EPSG: 4326, Figure \ref{fig: lake_shape}. We produced a two-dimensional triangular mesh using \textit{gmsh}, an open-source 3D finite element mesh generator that provide a fast, lightweight for mesh generation (\cite{geuzaine2009gmsh}) ensuring element diameters around 50m on average. The boundary $\partial\Omega$ thus represents the lake perimeter, where we impose Neumann conditions. Real wind data $(u(t),\,v(t))$ were measured at hourly intervals. We aggregated these into a daily time series, then applied an Akima interpolation to obtain a smooth $\vec{v}(t)$ for each PDE time step. In practice, we clamp the change of $\|\vec{v}(t)\|$ to a maximum of 5 m/s if outliers appear. For each $\Delta t$ in the solver, we update $\vec{v}(t)$ accordingly. This vector is assumed uniform in space. We use the Firedrake library (\cite{FiredrakeUserManual}) to implement the aforementioned weak forms. We let $V$ be a continuous Galerkin space of piecewise polynomials. We form the mixed function space $W = V\times V\times V\times V$, then defined
\[
U = (B,\,p,\,P,\,Q)\in W,\quad 
v = (v_1,\,v_2,\,v_3,\,v_4)\in W.
\]
We write the residual forms $F_1,\dots,F_4$ as in section \ref{sec:modelformulation}, each integrated over $\Omega$. The Neumann boundary condition is automatically enforced. We assemble these forms into a global problem $F(U)=0$ in Firedrake,
\begin{align*}
F(U) \;=\; \sum_{i=1}^4 F_i(B,p,P,Q),
\end{align*}
and solve with a backward Euler time discretization. Our code sets up a \textit{NonlinearVariationalSolver} with an \textit{ILU} or \textit{bjacobi} preconditioner. After each time step, we update $U_n \leftarrow U$ and proceed until final time $T_{\mathrm{final}}$. Finally, we tested the mesh resolution by halving average triangle diameters, confirming that solutions of $B,\,p,\,P$ converge within $\approx2\%$ difference (in $L^2(\Omega)$ norm) by $t=50$ days. We also verified that artificially removing wind advection reproduced earlier purely diffusive solutions. The full-length animations of each state variable can be found in the supporting information.

\subsection{Sobol Indices}
For the Sobol Indices, the first-order index for factor $f$ at time $t$ is defined as
\[
S1_f(t) \;=\; \frac{\mathrm{Var}\bigl[\mathbb{E}[\,B(t)\mid X_f]\bigr]}{\mathrm{Var}\bigl[B(t)\bigr]},
\]
while the total-order index is given by
\[
ST_f(t) \;=\; 1 \;-\; \frac{\mathrm{Var}\Bigl[\mathbb{E}_{\sim f}\bigl[B(t)\mid f\bigr]\Bigr]}{\mathrm{Var}\bigl[B(t)\bigr]},
\]
where $\mathbb{E}[\,B(t)\mid X_f]$ is the conditional expectation of $B(t)$ given $X_f$ (the factor of interest), and $\mathbb{E}_{\sim f}[B(t)\mid f]$ denotes the expectation with respect to all parameters except $f$.

In our implementation, we generated 2048 samples for each analysis using Saltelli's extension of the Sobol sequence through the \texttt{SALib} Python library (\cite{herman2017salib}). Each sampled configuration provides a unique set of parameter values. We considered the following ranges: $z_m \in [2.0,\; 10.0]$, $K_{bg} \in [0.1,\; 1.0]$, $D \in [0.01,\; 0.1]$, $p_{in} \in [0.0,\; 0.3]$, $\beta_B \in [0.01,\; 0.1].$ Each parameter set was used to initialize and solve our model for the entire time window, storing the resulting spatiotemporal evolution of $B(x,t)$. We then computed $S1_f(t)$ and $ST_f(t)$ across different time intervals and averaged over the spatial grid to assess how each parameter shapes the predicted cyanobacteria distribution. Additionally, to quantify how much variability in sensitivity arises spatially, we calculated the standard deviation of $(S1_f)_\ell$ and $(ST_f)_\ell$ over $\ell=1,\dots, N_x N_y$, which provided the error bars in our final bar plots. Therefore, the mean impact of each parameter on cyanobacteria growth and the extent to which that impact varies across different spatial locations in the domain. For plotting purposes, we split the solution into smaller time bins of approximately 60 days to track how the relative influences of these parameters vary throughout the simulation.

\section{Supplemental Figures}

\begin{figure}[ht]
    \centering
    \includegraphics[width=\linewidth]{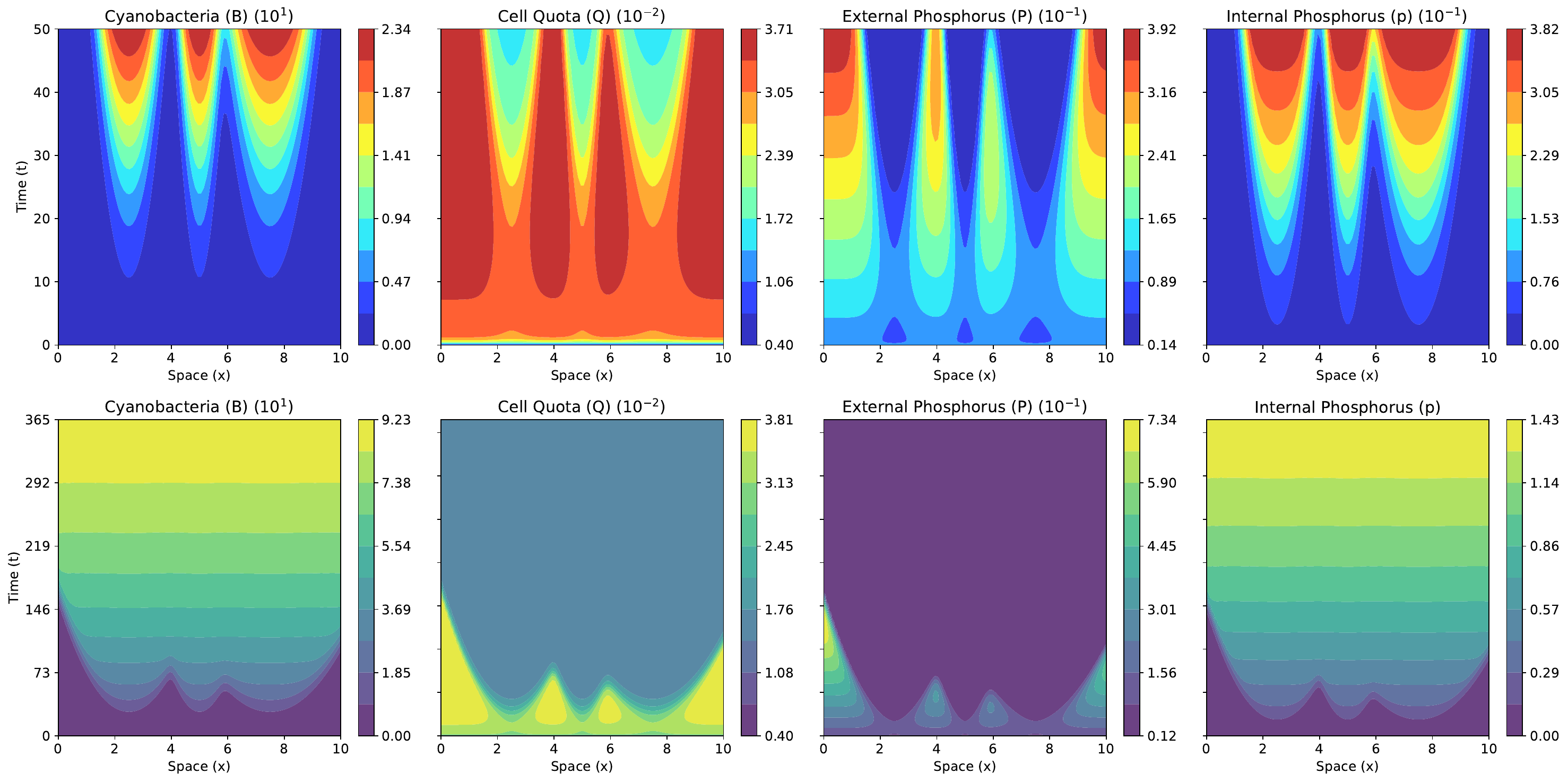}
    \caption{Time series simulations in one dimension of cyanobacteria, cell quota, external and internal phosphorus over short (50 days) and long (365 days) time periods. TThere was no wind/movement incorporated for the solution.}
    \label{fig:1D_Time_Series_No_Wind}
\end{figure}

\begin{figure}[ht]
    \centering
    \includegraphics[width=\linewidth]{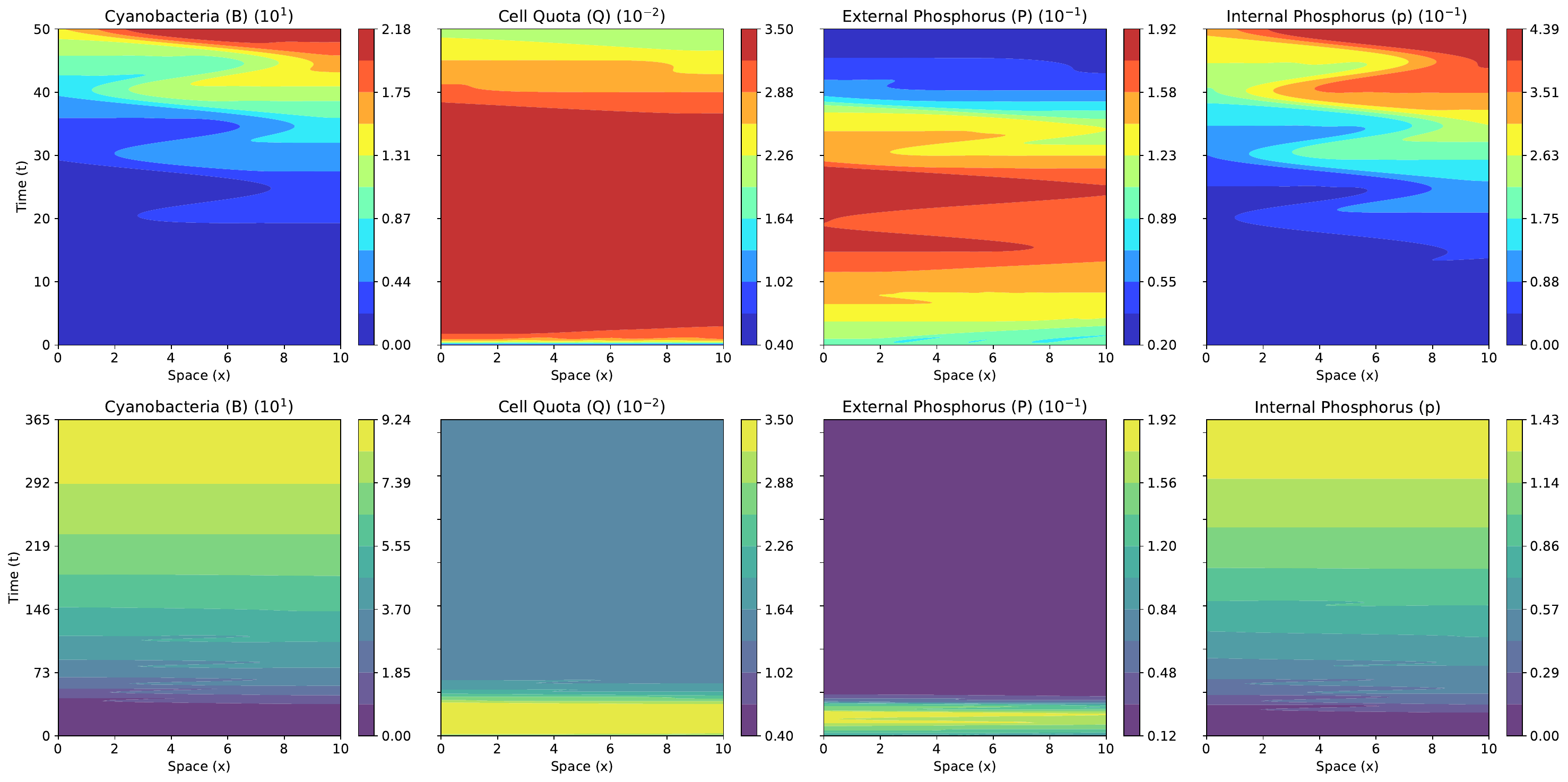}
    \caption{Time series simulations in one dimension of cyanobacteria, cell quota, external and internal phosphorus over short (50 days) and long (365 days) time periods. We used an oscillatory wind function to simulate periodic movement.}
    \label{fig:1D_Time_Series_Fake_Wind}
\end{figure}

\begin{figure}[ht]
    \centering
    \includegraphics[width=\linewidth]{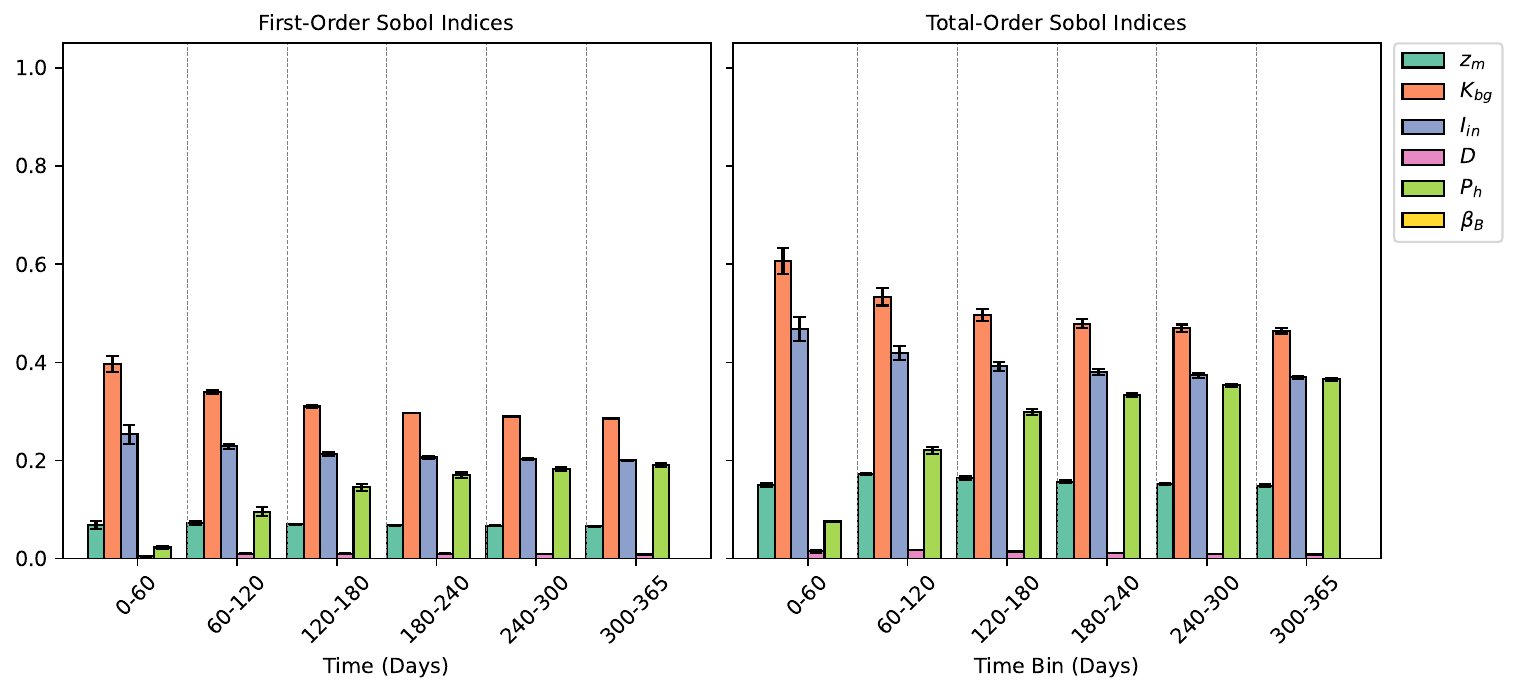}
    \caption{Global sensitivity analysis using first- and total-order Sobol indices to compare the impact of physical lake conditions on the growth of cyanobacteria over space and time. The simulation is for 365 days, with contributions binned into ~60 day intervals. There was no wind/movement incorporated for the solution.}
    \label{fig: sobol_no_wind}
\end{figure}

\begin{figure}[ht]
    \centering
    \includegraphics[width=\linewidth]{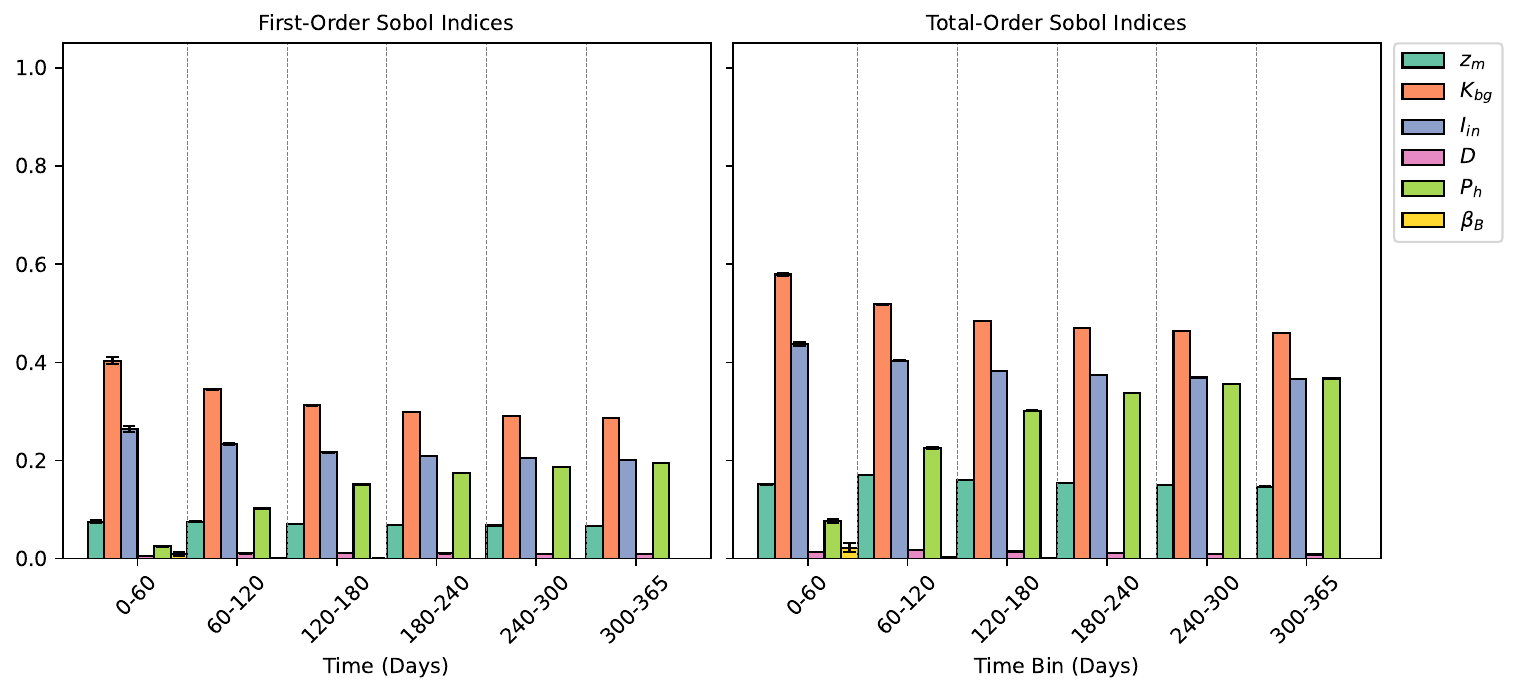}
    \caption{Global sensitivity analysis using first- and total-order Sobol indices to compare the impact of physical lake conditions on the growth of cyanobacteria over space and time. The simulation is for 365 days, with contributions binned into ~60 day intervals. We used an oscillatory wind function to simulate periodic movement.}
    \label{fig: sobol_fake_wind}
\end{figure}

\begin{figure}[ht]
    \centering  \includegraphics[width=0.7\linewidth]{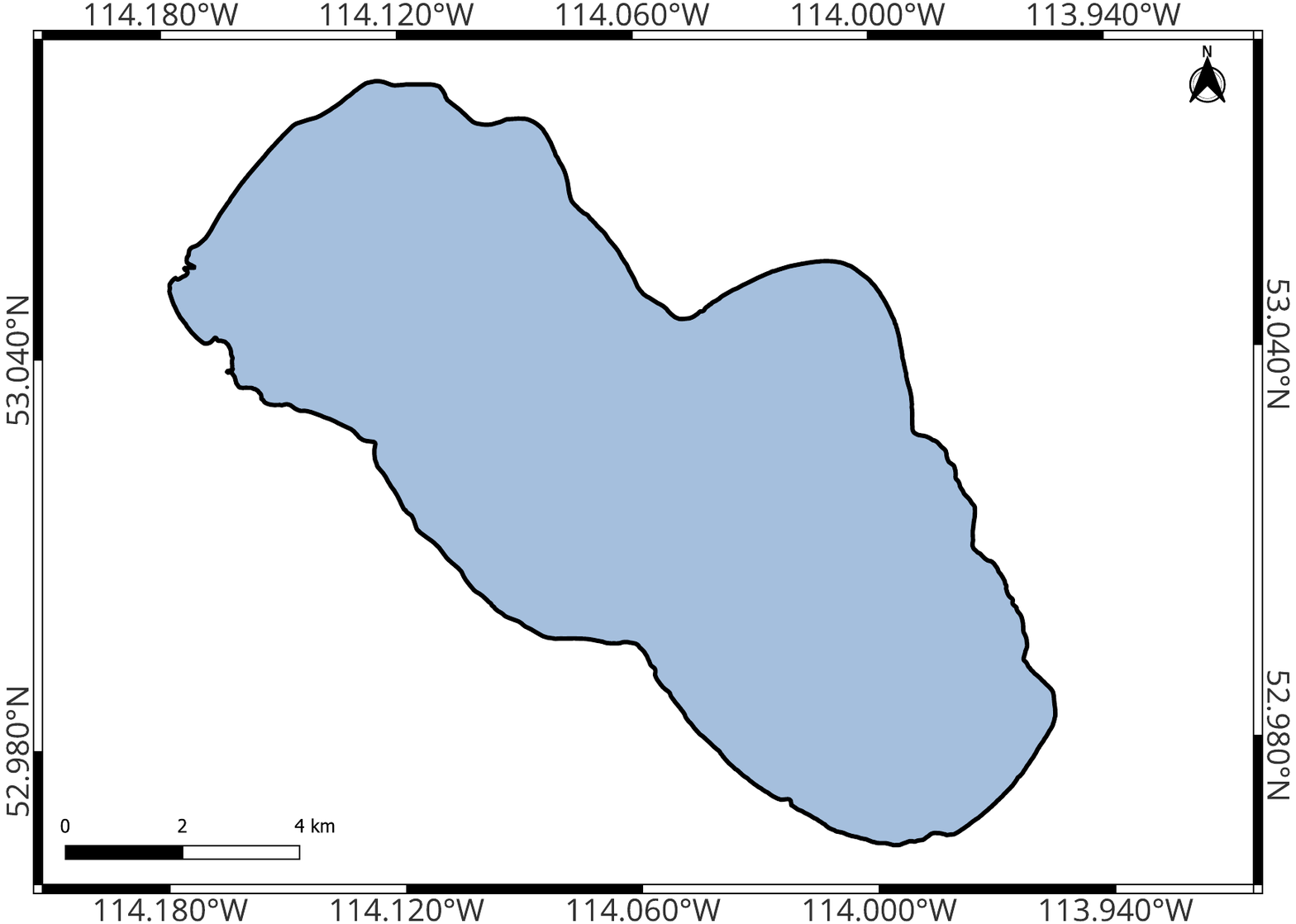}
    \caption{Closed curve representing the shape of Pigeon Lake, Alberta, Canada. Provided by the Government of Alberta, \cite{PigeonLake}}
    \label{fig: lake_shape}
\end{figure}

\begin{figure}[ht]
    \centering   \includegraphics[width=0.9\linewidth]{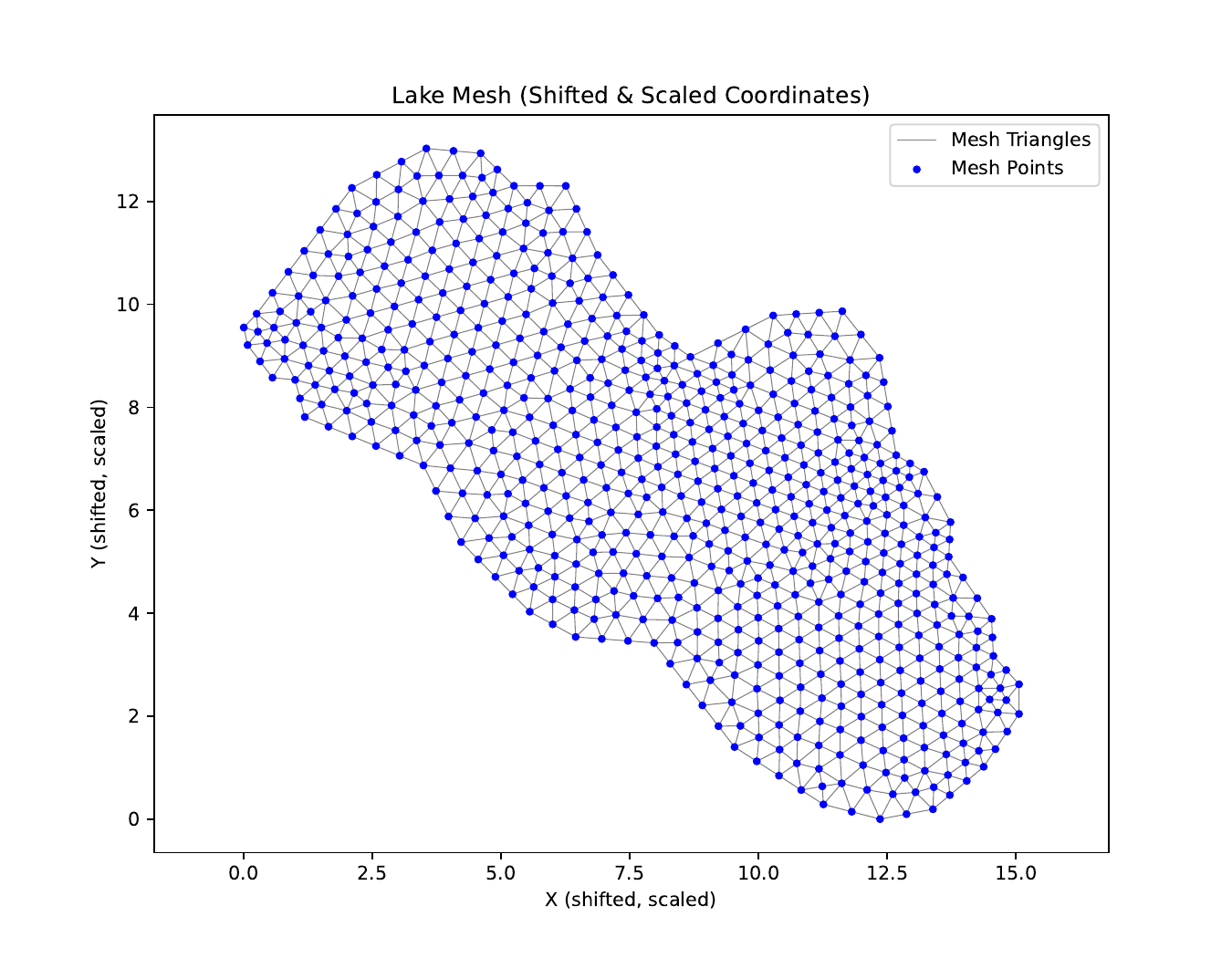}
    \caption{Mesh of Pigeon Lake, Alberta, Canada that we solved our PDE system on.}
    \label{fig: lake_shape_mesh}
\end{figure}

\clearpage

\printbibliography

@book{evans2010,
  author    = {Evans, Lawrence C.},
  title     = {Partial Differential Equations},
  edition   = {2},
  year      = {2010},
  publisher = {American Mathematical Society},
  address   = {Berkeley, CA},
  note      = {Department of Mathematics, University of California},
  doi       = {http://dx.doi.org/10.1090/gsm/019}
}

@article{zhang2022,
  title={Total Phosphorus and Nitrogen Dynamics and Influencing Factors in Dongting Lake Using Landsat Data},
  author={Zhang, Yuanyuan and Jin, Shuanggen and Wang, Ning and Zhao, Jiarui and Guo, Hongwei and Pellikka, Petri},
  journal={Remote Sensing},
  volume={14},
  number={22},
  pages={5648},
  year={2022},
  publisher={MDPI},
  doi = {https://doi.org/10.3390/rs14225648}
}

@article{qin2009,
author = {Qin, Qian and Clark, Jeff and Voller, Vaughan and Stefan, Heinz},
year = {2009},
month = {03},
pages = {},
title = {Depth-Dependent Dispersion Coefficient for Modeling of Vertical Solute Exchange in a Lake Bed under Surface Waves},
volume = {135},
journal = {Journal of Hydraulic Engineering},
doi = {https://doi.org/10.1061/(ASCE)0733-9429(2009)135:3(187)}
}

@article{gao2020,
  title={Remote sensing algorithms for estimation of fractional vegetation cover using pure vegetation index values: A review},
  author={Gao, Lin and Wang, Xiaofei and Johnson, Brian Alan and Tian, Qingjiu and Wang, Yu and Verrelst, Jochem and Mu, Xihan and Gu, Xingfa},
  journal={ISPRS Journal of Photogrammetry and Remote Sensing},
  volume={159},
  pages={364--377},
  year={2020},
  publisher={Elsevier},
  doi = {https://doi.org/10.1016/j.isprsjprs.2019.11.018.}
}

@article{du2022,
  title={Remote Estimation of the Particulate Phosphorus Concentrations in Inland Water Bodies: A Case Study in Hongze Lake},
  author={Du, Chenggong and Shi, Kun and Liu, Naisen and Li, Yunmei and Lyu, Heng and Yan, Chen and Pan, Jinheng},
  journal={Remote Sensing},
  volume={14},
  number={16},
  pages={3863},
  year={2022},
  doi = {https://doi.org/10.3390/rs14163863},
  publisher={MDPI}
}

@article{heggerud2020,
author = {Heggerud, Christopher M. and Wang, Hao and Lewis, Mark A.},
title = {Transient Dynamics of a Stoichiometric Cyanobacteria Model via Multiple-Scale Analysis},
journal = {SIAM Journal on Applied Mathematics},
volume = {80},
number = {3},
pages = {1223-1246},
year = {2020},
doi = {https://doi.org/10.1137/19M1251217},
}

@book{Cantrell2004,
  author    = {Robert Stephen Cantrell and Chris Cosner},
  title     = {Spatial Ecology via Reaction-Diffusion Equations},
  year      = {2004},
  isbn      = {9780471493013},
  publisher = {John Wiley \& Sons, Ltd},
  doi = {https://doi.org/10.1002/0470871296},
  url       = {https://doi.org/10.1002/0470871296}
}

@article{huisman1994,
  title={Light limited growth and competition for light in well mixed aquatic environments: An elementary model},
  author={Huisman, J. and Weissing, F.J.},
  journal={Ecology},
  volume={75},
  number={2},
  pages={507--520},
  year={1994},
  publisher={JSTOR},
  doi={https://doi.org/10.2307/1939554}
}

@article{zhao2014,
title = {Models for identifying significant environmental factors associated with cyanobacterial bloom occurrence and for predicting cyanobacterial blooms},
journal = {Journal of Great Lakes Research},
volume = {40},
number = {2},
pages = {265-273},
year = {2014},
issn = {0380-1330},
doi = {https://doi.org/10.1016/j.jglr.2014.02.011},
url = {https://www.sciencedirect.com/science/article/pii/S038013301400032X},
author = {Laijun Zhao and Wei Huang},
}

@article{cao2006,
  title={Effects of Wind and Wind-Induced Waves on Vertical Phytoplankton Distribution and Surface Blooms of Microcystis aeruginosa in Lake Taihu},
  author={Cao, H.S. and Kong, F.X. and Luo, L.C. and Shi, X.L. and Yang, Z. and Zhang, X.F. and Tao, Y.},
  journal={Journal of Freshwater Ecology},
  volume={21},
  number={2},
  pages={231--238},
  year={2006},
  publisher={Taylor \& Francis},
  doi={https://doi.org/10.1080/02705060.2006.9664991}
}

@article{amann1990,
author = {Amann, Herbert and Crandall, M.G.},
year = {1990},
month = {01},
pages = {13-75},
title = {Dynamic theory of quasilinear parabolic equations. II. Reaction-diffusion systems},
volume = {3},
journal = {Differential and Integral Equations},
doi = {https://doi.org/10.57262/die/1371586185}
}

@article{huisman2018cyanobacterial,
  title = {Cyanobacterial blooms},
  author = {Huisman, Jef and Codd, Geoffrey A. and Paerl, Hans W. and others},
  journal = {Nature Reviews Microbiology},
  volume = {16},
  pages = {471--483},
  year = {2018},
  doi = {https://doi.org/10.1038/s41579-018-0040-1}
}

@article{paerl2008blooms,
  title = {Blooms Like It Hot},
  author = {Paerl, Hans W. and Huisman, Jef},
  journal = {Science},
  volume = {320},
  pages = {57--58},
  year = {2008},
  doi = {https://doi.org/10.1126/science.1155398}
}

@article{paerl2009climate,
  title = {Climate change: a catalyst for global expansion of harmful cyanobacterial blooms},
  author = {Paerl, Hans W. and Huisman, Jef},
  journal = {Environmental Microbiology Reports},
  year = {2009},
  volume = {1},
  number = {1},
  pages = {27--37},
  doi = {https://doi.org/10.1111/j.1758-2229.2008.00004.x},
  pmid = {23765717}
}

@article{wilhelm2011,
title = {The relationships between nutrients, cyanobacterial toxins and the microbial community in Taihu (Lake Tai), China},
journal = {Harmful Algae},
volume = {10},
number = {2},
pages = {207-215},
year = {2011},
issn = {1568-9883},
doi = {https://doi.org/10.1016/j.hal.2010.10.001},
url = {https://www.sciencedirect.com/science/article/pii/S1568988310001150},
author = {Steven W. Wilhelm and Sarah E. Farnsley and Gary R. LeCleir and Alice C. Layton and Michael F. Satchwell and Jennifer M. DeBruyn and Gregory L. Boyer and Guangwei Zhu and Hans W. Paerl}
}

@article{kotak1996microcystin,
author = {Kotak, Brian and Zurawell, Ron and Prepas, Ellie and Holmes, Charles},
year = {1996},
month = {01},
pages = {1974-1985},
title = {Microcystin-LR concentration in aquatic food web compartments from lakes of varying trophic status},
volume = {53},
journal = {Canadian Journal of Fisheries and Aquatic Sciences},
doi = {https://doi.org/10.1139/cjfas-53-9-1974}
}

@article{sobol2001global,
  title     = {Global sensitivity indices for nonlinear mathematical models and their Monte Carlo estimates},
  author    = {Sobol, I. M.},
  journal   = {Mathematics and Computers in Simulation},
  volume    = {55},
  number    = {1--3},
  pages     = {271--280},
  year      = {2001},
  publisher = {Elsevier},
  doi = {https://doi.org/10.1016/S0378-4754(00)00270-6}
}

@misc{herman2017salib,
  title     = {{SALib}: an open-source {Python} library for sensitivity analysis},
  author    = {Herman, John and Usher, Will},
  howpublished = {\url{https://salib.github.io/SALib/}},
  note      = {Version 1.3.0},
  year      = {2017}
}

@article{Hardy2015,
  author = {Hardy, F. J. and Johnson, A. and Hamel, K. and others},
  title = {Cyanotoxin Bioaccumulation in Freshwater Fish, Washington State, USA},
  journal = {Environmental Monitoring and Assessment},
  year = {2015},
  volume = {187},
  number = {11},
  pages = {667},
  doi = {https://doi.org/10.1007/s10661-015-4875-x}
}

@article{wang2007,
author = {Wang, Hao and Smith, Hal L. and Kuang, Yang and Elser, James J.},
title = {Dynamics of Stoichiometric Bacteria-Algae Interactions in the Epilimnion},
journal = {SIAM Journal on Applied Mathematics},
volume = {68},
number = {2},
pages = {503-522},
year = {2007},
doi = {https://doi.org/10.1137/060665919},

URL = { 
    
        https://doi.org/10.1137/060665919
},
eprint = { 
    
        https://doi.org/10.1137/060665919
}
}

@article{wang2025existence,
  title={Existence and asymptotic stability of a generic Lotka-Volterra system with nonlinear spatially heterogeneous cross-diffusion},
  author = {Wang, Tianxu and Sim, Jiwoon and Wang, Hao},
  year = {2025},
  month = {01},
  pages = {},
  journal={arXiv preprint arXiv:2501.12569},
  doi = {https://doi.org/10.48550/arXiv.2501.12569}
}

@article{paerl2013harmful,
  title = {Harmful cyanobacterial blooms: causes, consequences, and controls},
  author = {Paerl, H. W. and Otten, T. G.},
  journal = {Microbial Ecology},
  volume = {65},
  number = {4},
  pages = {995--1010},
  year = {2013},
  publisher = {Springer},
  doi = {https://doi.org/10.1007/s00248-012-0159-y}
}

@book{reynolds2006,
  title = {The Ecology of Phytoplankton},
  author = {Reynolds, C. S.},
  year = {2006},
  publisher = {Cambridge University Press}
}

@article{wang2022,
title = {Mathematical comparison and empirical review of the Monod and Droop forms for resource-based population dynamics},
journal = {Ecological Modelling},
volume = {466},
pages = {109887},
year = {2022},
issn = {0304-3800},
doi = {https://doi.org/10.1016/j.ecolmodel.2022.109887},
url = {https://www.sciencedirect.com/science/article/pii/S030438002200014X},
author = {Hao Wang and Pablo Venegas Garcia and Shohel Ahmed and Christopher M. Heggerud}
}

@manual{FiredrakeUserManual,
  title        = {Firedrake User Manual},
  author       = {David A. Ham and Paul H. J. Kelly and Lawrence Mitchell and Colin J. Cotter and Robert C. Kirby and Koki Sagiyama and Nacime Bouziani and Sophia Vorderwuelbecke and Thomas J. Gregory and Jack Betteridge and Daniel R. Shapero and Reuben W. Nixon-Hill and Connor J. Ward and Patrick E. Farrell and Pablo D. Brubeck and India Marsden and Thomas H. Gibson and Miklós Homolya and Tianjiao Sun and Andrew T. T. McRae and Fabio Luporini and Alastair Gregory and Michael Lange and Simon W. Funke and Florian Rathgeber and Gheorghe-Teodor Bercea and Graham R. Markall},
  organization = {Imperial College London and University of Oxford and Baylor University and University of Washington},
  edition      = {First edition},
  year         = {2023},
  month        = {5},
  doi          = {https://doi.org/10.25561/104839},
}

@misc{PigeonLake,
  author = {Alberta Government},
  title = {Pigeon Lake, Alberta - Boundary (GIS data, polygon features)},
  year = {2016},
  publisher = {Alberta Energy Regulator},
  howpublished = {\url{https://open.alberta.ca/dataset/gda-dig_2008_0828}},
  note = {Accessed: 2025-02-11}
}

@article{geuzaine2009gmsh,
  title={Gmsh: A 3-D finite element mesh generator with built-in pre-and post-processing facilities},
  author={Geuzaine, Christophe and Remacle, Jean-Fran{\c{c}}ois},
  journal={International journal for numerical methods in engineering},
  volume={79},
  number={11},
  pages={1309--1331},
  year={2009},
  publisher={Wiley Online Library},
  doi = {https://doi.org/10.1002/nme.2579}
}

@article{hsu2010,
title = {On a system of reaction–diffusion equations arising from competition with internal storage in an unstirred chemostat},
journal = {Journal of Differential Equations},
volume = {248},
number = {10},
pages = {2470-2496},
year = {2010},
issn = {0022-0396},
doi = {https://doi.org/10.1016/j.jde.2009.12.014},
url = {https://www.sciencedirect.com/science/article/pii/S0022039609004677},
author = {Sze-Bi Hsu and Jifa Jiang and Feng-Bin Wang}
}

@article{hsu2014,
author = {Sze-Bi Hsu and Junping Shi and Feng-Bin Wang},
title = {Further studies of a reaction-diffusion system for an unstirred chemostat with internal storage},
journal = {Discrete and Continuous Dynamical Systems - B},
volume = {19},
number = {10},
pages = {3169-3189},
year = {2014},
issn = {1531-3492},
doi = {https://doi.org/10.3934/dcdsb.2014.19.3169},
url = {https://www.aimsciences.org/article/id/03ac9b63-170f-4867-8116-22cdbed1c74f},
author = {Sze-Bi Hsu and Junping Shi and Feng-Bin Wang}
}

@article{Hsu2017,
  author    = {Sze-Bi Hsu, Shui Bing and Lam, King-Yeung and Wang, Feng-Bin},
  title     = {Single species growth consuming inorganic carbon with internal storage in a poorly mixed habitat},
  journal   = {Journal of Mathematical Biology},
  year      = {2017},
  volume    = {75},
  number    = {6-7},
  pages     = {1775--1825},
  month     = dec,
  doi       = {https://doi.org/10.1007/s00285-017-1134-5},
  pmid      = {28497245},
  note      = {Epub 2017 May 11}
}

@article{zhang2021wind,
  author = {Zhang, Y. and Loiselle, S. and Shi, K. and Han, T. and Zhang, M. and Hu, M. and Jing, Y. and Lai, L. and Zhan, P.},
  title = {Wind Effects for Floating Algae Dynamics in Eutrophic Lakes},
  journal = {Remote Sensing},
  volume = {13},
  number = {4},
  pages = {800},
  year = {2021},
  doi = {https://doi.org/10.3390/rs13040800}
}
\end{document}